\documentclass[12pt,hyperref]{article}% 字体大小
\usepackage{graphicx} %% Package for Figure
\usepackage{float} %% Package for Float
\usepackage{amssymb}
\usepackage{amsmath}
\usepackage{mathtools}
\usepackage[thmmarks,amsmath]{ntheorem} %% If amsmath is applied, then amsma is necessary
\usepackage{lineno}
\usepackage{caption}
\usepackage{subfigure}
\usepackage{bm} %% Bold Mathematical Symbols
\usepackage{extarrows}
\usepackage{mathrsfs} %% Swash letter
\usepackage{booktabs} %% tables with three lines
{%% Environment of Proof
    \theoremstyle{nonumberplain}
    \theoremheaderfont{\bfseries}
    \theorembodyfont{\normalfont}
    \theoremsymbol{\mbox{$\Box$}}
    \newtheorem{proof}{Proof}
}

\usepackage{theorem}
\newtheorem{theorem}{Theorem}[section]
\newtheorem{proposition}{Proposition}[section]
\newtheorem{lemma}{Lemma}[section]

\newtheorem{definition}{Definition}[section]

\newtheorem{corollary}{Corollary}[section]
\newtheorem{claim}{Claim}[section]
\newtheorem{property}{Property}[section]

\newtheorem{conjecture}{Conjecture}[section]
{%% Environment of Remark
    \theoremheaderfont{\bfseries}
    \theorembodyfont{\normalfont}
    
}
\newcommand{\RNum}[1]{\uppercase\expandafter{\romannumeral #1\relax}}

\graphicspath{{figure/}}                                 %% Path of Figures
\usepackage[a4paper]{geometry}                        %% Paper size
\begin{document}
\title{\bf  Bipartitions with prescribed order of highly connected digraphs\footnote{The author's work is supported by National Natural Science Foundation of China, Grant/Award Number: 12071260.}}
\date{}
\author{\sffamily Yuzhen Qi, Jin Yan\footnote{Corresponding author. E-mail address: yanj@sdu.edu.cn.}, Jia Zhou\\
    {\sffamily\small School of Mathematics, Shandong University, Jinan 250100, China }}
    %{\sffamily\small $^2$ The second Institution }}
%\renewcommand{\thefootnote}{\fnsymbol{footnote}}
%\footnotetext[1]{Corresponding author. E-mail: yanj@sdu.edu.cn.}
%\footnotetext[2]{Corresponding author. E-mail: yanj@sdu.edu.cn.}
\maketitle

%MS+++++++++++++++++++++ Abstract +++++++++++++++++++++++++
{\noindent\small{\bf Abstract:}
 A digraph is strongly connected if it has a directed path from $x$ to $y$ for every ordered pair of distinct vertices $x, y$ and it is strongly $k$-connected if it has at least $k+1$ vertices and remains strongly connected when we delete any set of at most $k-1$ vertices.  For a digraph $D$, we use $\delta(D)$ to denote $\mathop{\text{min}}\limits_{v\in V (D)} {|N_D^+(v)\cup N_D^-(v)|}$. In this paper, we show the following result. Let $k, l, n, n_1, n_2 \in \mathbb{N}$ with $n_1+n_2\leq n$ and $n_1,n_2\geq n/20$. Suppose that $D$ is a strongly $10^7k(k+l)^2\log(2kl$)-connected digraph of order $n$ with $\delta(D)\geq n-l$. Then there exist two disjoint subsets $V_1, V_2\in V(D)$ with $|V_1| = n_1$ and $|V_2| = n_2$ such that each of $D[V_1]$, $D[V_2]$, and
$D[V_1, V_2]$ is strongly $k$-connected. In particular, $V_1$ and $V_2$ form a partition of $V(D)$ when $n_1+n_2=n$. This result improves the earlier result of Kim, K\"{u}hn, and Osthus [SIAM J. Discrete Math. 30 (2016) 895--911].

\vspace{1ex}
{\noindent\small{\bf Keywords: }digraph, connectivity, partition}

\vspace{1ex}
{\noindent\small{\bf AMS subject classifications.} 05C20, 05C40}

%MS++++++++++++++++++++++++++++++ Main body ++++++++++++++++++++
\section{Introduction}

Given a digraph  $D$, we write $D[U]$ for the subgraph of $D$ induced by a vertex set $U$ of $D$, and for two disjoint vertex sets $A,B$ of $D$, we use $D[A, B]$ to denote the subgraph of $D$, which consists of all arcs between $A$ and $B$ in $D$ but no others and $V(D[A, B])=A\cup B$. A digraph is \emph{strongly connected} if it has a directed path from $x$ to $y$ for every ordered pair of distinct vertices $x, y$ and it is $strongly$ $k$-$connected$ if it has at least $k+1$ vertices and remains strongly connected when we delete any set of at most $k- 1$ vertices.

Recently, researchers have shown an increased interest in the partitions of (di)graphs into subgraphs which inherit some properties of the original (di)graph.
  Thomassen \cite{reid(1989)} asked whether for integers $k_1,\ldots,k_t$, there exists $f(k_1,\ldots,k_t)$ such that every strongly $f(k_1, \ldots,k_t)$-connected tournament $T$ admits a vertex-partition $W_1, \ldots, W_t$ such that $T[W_i]$ is a strongly $k_i$-connected subtournament for each $i\in [t]$?

K\"{u}hn, Osthus and Townsend \cite{Kuhn(2016)} answered the above
problem by showing that every strongly $10^7k^6t^3$log$(kt^2)$-connected tournament $T$ admits a vertex-partition $V_1, \ldots, V_t$ such that $T[V_i]$ is a strongly $k$-connected subtournament for each $i\in [t]$. In 2020, Kang and Kim  \cite{Kang(2020)} improved this result in
the following theorem. Here, $\delta(D)=\mathop{\text{min}}\limits_{v\in V (D)} {|N_D^+(v)\cup N_D^-(v)|}$ and $\log k = \log_2 k$ for $k\in \mathbb{N}$.

\begin{theorem}\label{theorem4}
\cite{Kang(2020)} Let $k, t,l, m, n, q, a_1,\ldots, a_t \in \mathbb{N}$ with $t, m \geq 2$, $\sum_{i\in [t]} a_i \leq n$ and $a_i\geq n/(10tm)$ for each $i \in [t]$. Suppose that D is a strongly $10^8qk^2l(k+l)^2tm^2\log(m)$-connected digraph of order n  with $\delta(D)\geq n-l$, and $Q_1,\ldots, Q_t \subseteq V (D)$ are t disjoint sets with $|Q_i|\leq q$ for each $i\in [t]$. Then there exist disjoint sets $W_1, \ldots, W_t$ of $V(D)$ satisfying the following.

\noindent(1) $Q_i \subseteq W_i$.

\noindent (2) For every $i \in [t]$, the subgraph $D[W_i]$ is strongly $k$-connected.

\noindent (3) $|W_i| = a_i$.
\end{theorem}
Recently, Gir\~{a}o and Letzter \cite{An(2022)} showed that there exists a constant $c>0$ such that the vertices of every strongly $c\cdot kt$-connected tournament can be partitioned into $t$ parts, each of which induces a strongly $k$-connected tournament.

On the other hand, Kim, K\"{u}hn and Osthus \cite{Kim(2016)} proved that an even stronger type of bipartition can be obtained.
 \begin{theorem}\label{theorem2}
\cite{Kim(2016)} Let $k \in \mathbb{N}$, and let T be a strongly $10^9k^6\log(2k)$-connected tournament. Then  there exists a partition $V_1$, $V_2$ of $V(T)$ such that each of $T [V_1]$, $T [V_2]$, and
$T [V_1, V_2]$ is strongly $k$-connected.
\end{theorem}

Inspired by Theorems \ref{theorem4} and \ref{theorem2}, we prove the following result.

\begin{theorem}\label{theorem1}
 Let $k, l, n, n_1, n_2 \in \mathbb{N}$ with $n_1+n_2\leq n$ and $n_1,n_2\geq n/20$. Suppose that D is a strongly $10^7k(k+l)^2 \log(2kl)$-connected digraph of order n with $\delta(D)\geq n-l$. Then there exist two disjoint subsets $V_1, V_2$ of $V(D)$ satisfying the following.

\noindent (1) $|V_i| = n_i$ for each $i\in [2]$.

\noindent(2) Each of $D[V_1]$, $D[V_2]$, and $D[V_1, V_2]$ is strongly $k$-connected.

\noindent In particular, $V_1$ and $V_2$ form a partition of $V(D)$ when $n_1+n_2=n$.
\end{theorem}

Camion Theorem \cite{Camion(1959)} states that every strongly connected tournament contains a Hamiltonian cycle. In 1966, Moon \cite{Moon(1966)} extended this result by showing that  for any strongly connected tournament $T$ and a vertex $v \in V (T)$, there exists a cycle $C$ of length $ 3\leq |C|\leq |T|$ such that $v\in V(C)$. Moreover, we prove the following result.

\begin{corollary}\label{corollary}
Let $n,t\in \mathbb{N}$  with $n\geq 6 $ and $3 \leq t \leq n- 3$, and let  $T$ be a strongly $4\cdot 10^7$-connected tournament of order n with a vertex $v\in V(T)$. Then there exist two disjoint cycles $C_1$ and $C_2$ in $T$ such that $v\in V(C_1)$, $|C_1|=t$ and $|C_2|=n-t$.
\end{corollary}
\begin{proof}
By Theorem \ref{theorem1}, there exists a partition  $V_1$, $V_2$ of $V(T)$ such that each of $T [V_1]$, $T [V_2]$, and
$T [V_1, V_2]$ is strongly connected. One may assume $v\in V_1$. If $|V_1|\geq t$, then by Moon Theorem,  $T [V_1]$ has a  cycle $C_1$ of length $t$ and $v\in V(C_1)$. Because  $T [V_1, V_2]$ is strongly connected, each vertex $u\in V_1\setminus V(C_1)$ has an in-neighbor and an out-neighbor in $V_2$. So, $V(T)\setminus V(C_1)$ is strong. Owing to Camion Theorem, we are done. Otherwise, Moon Theorem shows that $T [V_2]$ has a cycle $C_2$ of length $n-t$. Similarly, $V(T)\setminus V(C_2)$ is strong, and then $T[V(T)\setminus V(C_2)]$ has a  Hamiltonian cycle.
\end{proof}

Partitioning an undirected graph into several parts under connectivity constraint has been widely studied due to
its important applications. A number of classical results in undirected graphs have been achieved.  Surveys such as that conducted by Hajnal \cite{Hajnal(1983)} and Thomassen \cite{Thomassen(1983)} proved that for every $k$ there exists an integer $f(k)$ such that every $f(k)$-connected graph has a vertex partition into sets $A$ and $B$ so that both $A$ and $B$ induce $k$-connected graphs. Later, K\"{u}hn and Osthus \cite{Kuhn(2003)} further strengthened this by requiring that each vertex in $A$ has many neighbours in $B$. %However, Kim, K\"{u}hn and Osthus \cite{Kim(2016)} gave an example to show that it is not hard to prove that a highly connected graph $G$  cannot admit a partition $V(G)=A\cup B$ such that each of  $G[A]$, $G[B]$, and $G[A, B]$ is $k$-connected.

The rest of the paper is organized as follows. Section 2 contains preliminaries and tools throughout the paper. The proof of Theorem \ref{theorem1} is given in Section 3, which consists of 4 subsections.

%More precisely, we construct the almost dominating sets in Subsection 3.1, and in Subsection 3.2, we try to find ``correct $i$-paths" (see Definition \ref{definition2}), and then we deal with some ``special" vertices in Subsection 3.3. At last, we get the required sets $V_1$ and $V_2$ in Subsection 3.4. The proof is similar in spirit to Kim, K\"{u}hn and Osthus \cite{Kim(2016)}. However, the way we find ``correct $i$-paths" and deal with ``special" vertices is different  from \cite{Kim(2016)}.
\section{Preliminaries and tools}
\subsection{Basic terminology}
 Given $k\in \mathbb{N}$, let $[k]= \{1, \ldots, k\}$ and $[k, k+l] = \{k,\ldots, k+l\}$. The digraphs considered in this paper are always simple (without loops and multiple arcs). Let $D$ be a digraph with vertex set $V(D)$ and arc set $A(D)$. We write $|D|$ for the number of vertices in $D$ and $xy$ for an arc directed from $x$ to $y$. Moreover, we call $x$  \emph{in-neighbor} of $y$ and $y$ \emph{out-neighbor} of $x$ or $x$ dominates $y$ and $y$ is dominated by $x$. For any vertex $x\in V(D)$, let
$N^+_D (x) = \{y: xy\in A(D)\}$ and $N^-_D (x) = \{y: yx\in A(D)\}$ be the \emph{out-neighborhood} and the \emph{in-neighborhood} of $x$, respectively. Similarly, define $N^{-!}_D (x) = \{y: yx\in A(D)\text{ and } xy\notin A(D) \}$ and $N^{+!}_D (x) = \{y: xy\in A(D)\text{ and } yx\notin A(D) \}$.
We call $d_D^+ (x) = |N_D^+(x)|$,  $d_D^- (x) = |N_D^-(x)|$, $d^{-!}_D (x) = |N^{-!}_D (x)|$ and $d^{+!}_D (x) = |N^{+!}_D (x)|$
 the \emph{out-degree}, the \emph{in-degree}, the \emph{sole in-degree}  and the \emph{sole out-degree} of $x$ in $D$, respectively. Further, let $\delta^-(D)$ and $\delta^+(D)$ denote the\emph{ minimum in-degree} and the \emph{minimum out-degree} of $D$.

Given a digraph $D$ and a path $P=v_1v_2\cdots v_l$ of $D$, we use $v_iPv_j$ to denote the subpath of $P$ from $v_i$ to $v_j$, and denote the set $V(P)\setminus \{v_1,v_{l}\}$ of internal vertices of $P$ by $\text{Int}(P)$.  The \emph{length} of $P$ is the number of its vertices.
A path $P$ is called $odd$ if its length is odd, and $even$ if its length is even.

Let $A$ and $B$ be two subsets of $V(D)$. We say that $A$ \emph{in-dominates} (resp. \emph{out-dominates}) $B$ if for every vertex $b\in B$ there exists a vertex $a\in A$ such that $ba\in A(D)$ (resp. $ab\in A(D)$). Further, we call $A$ an \emph{almost in-dominating set} (resp. \emph{almost out-dominating set}) if there exists a vertex subset $E$ of $V(D)$ such that $A$ in-dominates   (resp. out-dominates) $V(D)\setminus (A\cup E)$. In particular, if $E=\emptyset$, then  we call $A$ an \emph{in-dominating set} (resp. \emph{out-dominating set}). For a vertex set $U$ and a vertex $v$ in $D$, we say that $(v, U)$ is $k$\emph{-connected} in $D$ if for any subset $S \subseteq V (D) \setminus \{v\}$ with $|S|\leq k -1$, there exists a path from $v$ to a vertex in $U\setminus S$ in $D \setminus S$. Similarly, we say
$(U, v)$ is $k$\emph{-connected} in $D$ if for any subset $S\subseteq V (D) \setminus \{v\}$ with $|S| \leq k-1$, there exists
a path from a vertex in $U \setminus S$ to $v$ in $D\setminus S$. Note that $U$ does not have to be a subset of
$V(D)$ in this definition. However, if $(v, U)$ is $k$-connected in $D$ or $(U, v)$ is $k$-connected
in $D$, then either $v \in U$ or $|U \cap V (D)|\geq k$.

 %In particular, we say $P$ is \emph{transitive} if
%$v_iv_j\in A(T)$ if and only if $i < j$. In this case, we often say that $v_1$ is the \emph{tail} of $P$ and $v_l$ is the
%\emph{head} of $P$.

%We say that two paths are disjoint if they are vertex-disjoint.
%

%The following can be easily observed. We omit the proof.
%
%\begin{fact}\label{fact1}
%Let $n,l \in \mathbb{N}$, and let $D$ be a digraph with $\delta(D) \geq n - l$. Then D contains a vertex $u$ with $d^{-!}_D (u)\leq (n-1)/2$.% and a vertex v with  $d^{-}_D (v)\leq (n-l)/2$.
%\end{fact}
%  The proof of Theorem \ref{theorem1} is based on the concept of ``linkage structures in digraphs''., which play a key role to construct robust linkage structures

\subsection{Some lemmas}
We now collect some results that will be used frequently in the whole paper. The following two lemmas guarantee the existence of a suitable almost dominating set in a digraph close to being semicomplete. A tournament $T$ of order $n$ is \emph{transitive} if there exists an ordering $v_1, \ldots ,v_n$ of its vertices such that
$v_iv_j\in A(T)$ if and only if $i < j$. In this case, we often say that $v_1$ is the \emph{tail} of $T$ and $v_n$ is the
\emph{head} of $T$.

\begin{lemma}\label{lemma1}
Let $n,l,c \in \mathbb{N}$ with $c\geq 2$ and let $D$ be a digraph with $\delta(D) \geq n-l$. For any vertex $v\in V(D)$ with $d_D^-(v)\geq 2^{c-1}l$, there exist disjoint sets $A, E \subseteq V(D)$ and a vertex $a \in A$ such that the following properties hold:

\noindent (i) $2\leq |A| \leq c$, and $D [A]$ has  a spanning transitive tournament with tail $a$ and head $v$.

\noindent (ii) $A \setminus \{a\}$ out-dominates $V (D)\setminus (A\cup E)$.

\noindent (iii) $|E| \leq (1/2)^{c-2}d^-_D(v) + c(l-1)$.
\end{lemma}

\begin{proof}
Let $v_1:= v$. To this end, we will find $A$ by choosing vertices $v_1,\ldots ,v_i$ such that $v_i\in \bigcap\limits_{j\in [i-1]} N^{-}(v_j)$, and the order of union of the out-neighborhood of $v_j$  is as large as possible at each step. More precisely, suppose inductively that for some $1 \leq i < c$ we have already found a set $A_i = \{v_1,\ldots ,v_i\}$ and a set $E_i$ such that the following hold:

\noindent (a) $D[A_i]$ has a spanning transitive tournament with tail $v_i$ and head $v_1$;

\noindent (b)  $A_i$ out-dominates $V(D) \setminus (A_i \cup E_i)$, where $E_i = V(D)\setminus \bigcup\limits_{j\in[i]} N^{+}(v_i)$;

\noindent (c) $|E_i| \leq (1/2)^{i-1}d^-_D(v) + (i+1)(l-1)$.

Note that $A_1$ satisfies (a)-(c), since $E_1=N_D^{-}(v_1)$ and $|E_1| \leq d^-_D(v_1) + 2(l-1)$. Set $U_i= \bigcap\limits_{j\in [i]} N^{-}(v_j)$. We will either find desired sets $A$ and $E$, or extend $A_i$ by adding a new vertex in $U_i$. We see that if $|E_i| \leq (1/2)^{c-2}d^-_D(v) + c(l-1)$, then choose $a\in U_i$. Let $A: = A_i \cup \{a\}$ and $E: = E_i \setminus\{a\}$, as required. Hence $|E_i|> (1/2)^{c-2}d^-_D(v) + c(l-1)$. It follows from the definition of $U_i$ that $|U_i|> 2l$. Let $U_i^\prime$ be a subset of $U_i$ such that the in-degree of each vertex in $U_i^\prime$ is at least one. Obviously, $|U_i^\prime|\geq  l$ and $U_i^\prime$ has a vertex $u$ of sole in-degree at most $|U_i^\prime|/2$ in $D[U_i^\prime]$. Choose  $v_{i+1}:=u$.
Therefore,
\begin{equation*}
|E_{i+1}| = |V(D)\setminus \bigcup\limits_{j\in[i+1]} N^{+}(v_j)|
\leq (1/2)^{i}d^-_D(v)+(i+2)(l-1).
\end{equation*}
By repeating this construction, we will eventually gain $|E_i| \leq (1/2)^{c-2}d^-_D(v) + c(l-1)$  for some $i\leq c-1$. Then choose $a\in U_i$. Let $A := A_i \cup \{a\}$ and $E:= E_i \setminus\{a\}$. This completes the proof of Lemma \ref{lemma1}.
\end{proof}

The next lemma follows immediately from Lemma \ref{lemma1} by reversing the orientations of all arcs.

\begin{lemma}\label{lemma2}
 Let $n,l,c \in \mathbb{N}$ with $c\geq 2$ and let $D$ be a digraph with $\delta(D) \geq n-l$. For any vertex $v\in V(D)$ with $d_D^+(v)\geq 2^{c-1}l$, there exist disjoint sets $B, E \subseteq V(D)$ and a vertex $b \in B$ such that the following properties hold:

\noindent (i)   $2 \leq |B| \leq c$, and $D[B]$ has a spanning transitive tournament with tail $v$ and head $b$;

\noindent (ii)  $B\setminus \{b\}$ in-dominates $V(D) \setminus (B\cup E)$;

\noindent (iii) $|E| \leq (1/2)^{c-2}d^+_D(v) +c(l-1)$.
\end{lemma}

 A digraph $D$ is $k$\emph{-linked} if $|D|\geq 2k$
and, whenever $(x_1, y_1),\ldots , (x_k, y_k)$ are ordered pairs of (not necessarily distinct) vertices of $D$, there exist distinct internally disjoint paths $P_1,\ldots, P_k$
such that for all $i\in [k]$ we have that $P_i$ is a path from $x_i$ to $y_i$ and that $\{x_1, \ldots, x_k, y_1,\ldots,y_k\} \cap V (P_i) =\{x_i, y_i\}$. The following proposition follows immediately from the definition of linkedness and the pigeonhole principle.

\begin{proposition}\label{proposition2}
Let $k,s \in \mathbb{N}$, and let $D$ be a $ks$-linked digraph. Let $(x_1, y_1),\ldots , (x_k, y_k)$ be ordered pairs of (not necessarily distinct) vertices of $D$. Then the following two statements hold.

 \noindent(i) $D$ contains $ks$ distinct internally disjoint paths $P^1_1,\ldots,P_k^s$ such that for all $i\in [k]$ and $j\in [s]$ we have that
$P_i^j$  is a directed path from $x_i$ to $y_i$ and that $\{x_1, \ldots, x_k, y_1,\ldots,y_k\} \cap V (P_i^j) =\{x_i, y_i\}$.

 \noindent (ii) There exist $s^\prime\leq s$ internally disjoint paths such that the length of these paths is at most  $(s^\prime|D|)/s$.
\end{proposition}

It is easy to see that every $k$-linked digraph is strongly $k$-connected. The converse is not true. Meng, Rolek, Wang and Yu \cite{Meng(2021)} proved that every strongly $(40k -31)$-connected tournament is $k$-linked. Later, Bang-Jensen and Johansen \cite{Bang(2022)} improved their result by showing  that every strongly $(13k-6)$-connected tournament with minimum out-degree at least $28k-13$ is $k$-linked. Recently, Zhou, Qi and Yan \cite{Zhou(2018)} extended the result to a digraph by proving the following theorem.

%Analysis similar to that in the proof of Theorem \ref{theorem3} with obvious modifications using Lemma \ref{lemma2} implies the following corollary. %We omit the proof here, because it can be proved by the almost same proof as in \cite{Bang(2022)}

\begin{theorem}\label{corollary1}
\cite{Zhou(2018)} For each $n, k, l \in \mathbb{N}$, every strongly $(50k+20l)$-connected digraph with $\delta(D)\geq n-l$ is $k$-linked.
\end{theorem}

Lemma \ref{lemma3} states that there exist small (not necessarily disjoint) vertex sets $U$ and $W$ in a digraph such that every vertex in a digraph can reach $U$ and can be reached by a vertex in $W$. Moreover, Proposition \ref{proposition1} guarantees that a digraph has a small vertex set $U$  such that every vertex outside $U$ has many out- and in-neighbors in $U$.

\begin{lemma}\label{lemma3}
\cite{Kang(2020)} Let $n,k,l\in \mathbb{N}$ and let $D$ be a digraph of order $n$ with $\delta(D) \geq n -l$. Then there exist two sets $U, W \subseteq V (D)$ satisfying the following.

\noindent (i) $|U|,|W|\leq 2k+l-2$.

\noindent (ii) For every $v\in V(D)$, both $(v, U)$ and $(W, v)$ are $k$-connected in $D$.
\end{lemma}

%We also need the following lemma, which
%shows that there also exist small (not necessarily disjoint) sets $U$ and
%$W$ which are ``quickly reachable in a robust way" in a digraph.

%We frequently use the following proposition, which guarantees a small vertex set $U$ in a digraph such that every vertex outside $U$ has many out- and in-neighbors in $U$.

\begin{proposition}\label{proposition1}
Let $n, k, l \in \mathbb{N}$ and let $D$ be a digraph of order $n$ with $\delta(D) \geq n-l$. Then there is a set  $U\subseteq V (D)$ with $|U| \leq 3(k+l) \log n$ such that each vertex in $V(D)\setminus U$
has at least $k$ out-neighbors and at least $k$ in-neighbors in $U$.
\end{proposition}
\begin{proof}
First, we construct an almost in-dominating set $V_1$ of order at most $c=\lceil \log n \rceil\leq (3\log  n)/2$. Roughly speaking, we will find $V_1$ by choosing vertices $v_1,\ldots ,v_i$ such that $v_i\in \bigcap\limits_{j\in [i-1]} N_D^{+!}(v_j)$ with  minimum sole out-degree in $\bigcap\limits_{j\in [i-1]} N_D^{+!}(v_j)$ at each step for $i\in [c]$. We first choose $v_{1} \in V(D)$ with minimum sole out-degree in $D$. Set $A_1: = \{v_1\}$ and  $E_{1}:= N^{+!}_{D}(v_1)$. Obviously, $A_1$ is an almost in-dominating set. Hence  $|E_{1}|\leq \lceil n / 2 \rceil$. Then we choose vertices $v_2,\ldots ,v_i$ such that the sole out-degree of $v_i$ in $E_{i-1}$ is minimum at each step. Set $A_i: = A_{i-1} \cup \{v_i\}$ and  $E_{i}:=E_{i-1}\cap N^{+!}_{D}(v_j)$. Then  $A_i$ is an almost in-dominating set.  This yields $|E_{i}|\leq \lceil n/ 2^{i-1} \rceil$. Let $r$ be an integer such that $2^r \leq n <2^{r+1}$. Note that if $i$ is larger than $r+1$, then $E_i$ is empty and it not hard to see $r+1\geq \lceil \log n \rceil$. Therefore, let $V_1=A_{r+1}$ be the desired almost in-dominating set.

 We now apply this argument again, with $V(D)$ replaced by $V(D)\setminus (V_1\cup \cdots \cup V_j)$ with $j\in [k+l-1]$, to obtain $k+l$ disjoint almost in-dominating sets $V_1,\ldots,V_{k+l}$ with $|V_i|\leq c$ for all $i\in [k+l]$. Proceed similarly to obtain disjoint almost out-dominating sets $V_1^\prime,\ldots,V_{k+l}^\prime$ with $|V_i^\prime|\leq c$ for all $i\in [k+l]$ in $D - \bigcup_{i\in[k+l]}V_i$. Take $U:=V_1 \cup \cdots \cup V_{k+l}\cup V_1^\prime \cup \cdots \cup V_{k+l}^\prime$ as required.
\end{proof}

\section{Proof of Theorem \ref{theorem1}}

\textbf{Outline of the proof.} The idea of the proof of Theorem \ref{theorem1} goes back at least as far as \cite{Kim(2016)}. Our proof involves looking at constructing small disjoint out-dominating sets $A_i\setminus \{a_i\}$, where $i\in [6k]$, such that each $D[A_i]$ has a spanning transitive tournament with head $x_i$ and tail $a_i$. Similarly, we also need to construct small disjoint in-dominating sets $B_i\setminus \{b_i\}$, where $i\in [6k]$, so that each $D[B_i]$ has a spanning transitive tournament with head $b_i$ and tail $y_i$. Because $D$ is highly connected, provided that all these dominating sets are small enough, then we can use Proposition \ref{proposition2} (i) to find $6k$ disjoint paths $P_i$ joining the head $b_i$ of $B_i$ to the tail $a_i$ of $A_i$.

We see that if
\begin{equation}\label{18}
  \begin{aligned}
  &\bigcup\limits_ {i\in[k]} (A_i \cup B_i \cup V(P_i)) \cup \{a_{2k+1},b_{2k+1},b_{4k+1},a_{5k+1}\}\subseteq V_1, \bigcup\limits_ {i\in[k+1,2k]} (A_i\cup B_i \cup V(P_i))\\
  &\cup \{a_{3k+1},b_{3k+1},a_{4k+1},b_{5k+1}\}\subseteq  V_2 \text{ and } \bigcup\limits_ {i\in[2k+1,6k]} (A_i\cup B_i \cup V(P_i))\subseteq D[V_1,V_2],
  \end{aligned}
\end{equation}
then for any set $S$ with $|S|\leq k-1$, we can choose an index $i\in [6k]$ such that $(A_i\cup B_i \cup V(P_i))\cap S=\emptyset$. This implies that if $x, y \in V_1\setminus S$, then we can find a path
from $x$ to $y$ in $D[V_1\setminus S]$ by utilizing $P_i$. Hence $D[V_1]$ is strongly $k$-connected. In the same way, both $D[V_2]$ and  $D[V_1,V_2]$  are strongly $k$-connected.

However, we cannot hope for this ideal case, due to several issues. The following are some major issues complicating the proof.

$\bullet$  As out- or in-dominating sets may be large in general,  we can not use Proposition \ref{proposition2} (i) to get paths $P_1,\ldots,P_{6k}$.  Hence we will use Lemmas \ref{lemma1} and  \ref{lemma2} to construct small ``almost out-dominating sets" $A_i$ and ``almost in-dominating sets" $B_i$ in Subsection 3.1. This means that there is a small ``exceptional set'' of vertices that are not out-dominated (resp. in-dominated) by $A_i$ (resp. $B_i$). That is, a vertex in an ``exceptional set'' can not receive an arc from $A_i$ or send an arc to $B_i$. Instead of receiving/sending  an arc from $A_i$/to $B_i$ , we will use the notion of ``safe" vertices (see Definition \ref{definition1}) to find a path from  $A_i$ to a vertex in an ``exceptional set'' or from a vertex in an ``exceptional set'' to $B_i$.

$\bullet$ Owing to (\ref{18}), we obtain $P_{2k+1}\cup \cdots \cup P_{6k+1} \subseteq D[V_1,V_2]$. This implies that the length of $P_i$ is odd if $a_i$ and $b_i$ belong to the same partite. Otherwise, the length of $P_i$ is even. Hence we will use the notion of ``correct" path (see Definition \ref{definition2}) to make sure that the length of this path is even or odd. %Moreover, the paths $P_i$ that we construct will be either ``short" or ``long". We are very skillful to deal with the ``long" paths to get one ``correct" path, which is different from Kim, K\"{u}hn, and Osthus  \cite{Kim(2016)}. We will further elaborate on them in Subsection 3.2.

In the proof of Theorem \ref{theorem1}, we will construct the sets $V_1,V_2$ in several steps by coloring the vertices of $V(D)$ with colors $I$ and $II$. At
each step and for $i\in \{I,II\}$, let $V_{i}=\{v\in V(D): v \text{ is colored by } i\}$, and meanwhile we guarantee the safety of all colored vertices at this step.  Finally, let $V_1:=V_I$ and $V_{2}:=V_{II}$.
 The construction and coloring mainly consist of four steps:

$(a)$ Constructing $6k$  almost out-dominating sets $A_1,\ldots, A_{6k}$ and $6k$ almost in-dominating sets $B_1,\ldots, B_{6k}$ satisfy (P0)-(P4).  Then we can color all vertices in $A_1\cup \cdots\cup  A_{6k}\cup B_1\cdots \cup B_{6k}$ and some other vertices such that all colored vertices are safe (see Claim \ref{claim2}). Let $C_1$  be a set consisting of all vertices colored so far.

$(b)$  Finding correct $i$-paths for each $i\in [6k]$, coloring all vertices in correct $i$-paths and some other vertices such that all colored vertices are safe (see Claims \ref{claim3}-\ref{claim10}).  Let $C_5$ be the set consisting of all vertices colored so far. So $C_1\subseteq C_5$.

$(c)$  Coloring all remaining vertices such that all vertices are safe (see Claim \ref{claim11}). Let $C_6$ be the set consisting of all vertices colored so far. So $C_5\subseteq C_6$.

$(d)$  Constructing $V_1$ and $V_2$ (see Claim \ref{f}).

%  see Fig. \ref{fig6}. The thick arrows indicate the coloring order.
%\begin{figure}[H]
%\centering    %居中       %子图居中
%    \includegraphics[scale=0.65]{3}   % 以pic.jpg 的0.5倍大小输出
%\caption{The coloring steps in the proof of Theorem \ref{theorem1}.} %  %大图名称
%\label{fig6}  %图片引用标记
%\end{figure}

%To ensure this, we will prepare a small set $C^\prime$ and a coloring of the vertices in $C^\prime$ such that all colored vertices are ``safe" at this step (see Claim \ref{claim1}).
\subsection{Construction of almost dominating sets}

 Let $D$ be a strongly $10^7k(k+l)^2\log(2kl)$-connected digraph of order $n$ with $\delta(D)\geq n-l$. Let $X=\{x_1,x_2,\ldots,x_{6k}\}$ be a set of $6k$ vertices with smallest in-degrees in $D$, and  $Y=\{y_1,y_2,\ldots,y_{6k}\}$ be a set of $6k$ vertices with smallest out-degrees in $V(D)\setminus X$. Define
\begin{equation*}
\hat{\delta}^-(D)=\mathop{\text{min}}\limits_{v\in V(D)\setminus X}d_{D}^-(v)\ \ \ \  \text{ and }\ \ \ \  \hat{\delta}^+(D)=\mathop{\text{min}}\limits_{v\in V(D)\setminus Y}d_{D}^+(v).
\end{equation*}
Since $D$ is strongly $10^7k(k+l)^2\log(2kl)$-connected, it follows that
\begin{equation}\label{1}
 \text{min}\{n,\hat{\delta}^-(D), \hat{\delta}^+(D)\}\geq 10^7k(k+l)^2\log(2kl).
\end{equation}

Let $c=\lceil \text{log}(18000k^2)\rceil+2\leq 20k$. We first repeatedly apply Lemma  \ref{lemma1} (removing the dominating sets obtained already each time) with parameter $c$ and a vertex in $X$ to obtain almost out-dominating sets $A_1,\ldots, A_{6k}$ and sets of vertices $E_{A_1},\ldots, E_{A_{6k}}$. Proceed similarly (apply Lemma  \ref{lemma2} with parameter $c$ and a vertex in $Y$ by removing the dominating sets each time) to obtain disjoint almost in-dominating sets $B_1,\ldots, B_{6k}$, and sets of vertices $E_{B_1},\ldots, E_{B_{6k}}$. Write $D_0= \bigcup_{i\in[6k]}(A_i\cup B_i)$. We have the following statements.

(P0) The sets $A_1,\ldots, A_{6k}, B_1,\ldots, B_{6k}$ are pairwise disjoint. Moreover, $2 \leq |A_i|, |B_i|\leq c$.

 (P1) $D[A_i]$ has a spanning transitive tournament with head $x_i$ and tail $a_i$, and $D[B_i]$ has a spanning transitive tournament with tail
$y_i$ and head $b_i$.

 (P3) $A_i\setminus \{a_i\}$ out-dominates $V(D)\setminus (D_0\cup E_{A_i})$, and $B_i\setminus  \{b_i\}$ in-dominates $V(D)\setminus (D_0\cup E_{B_i})$.

(P4) $|E_{A_i}|\leq (1/2)^{c-2}\hat{\delta}^-(D)+(c-1)(l-1)$ and  $|E_{B_i}|\leq (1/2)^{c-2}\hat{\delta}^+(D)+(c-1)(l-1)$.

By symmetry, assume $\hat{\delta}^+(D)\geq \hat{\delta}^-(D)$. Let $E_A=\bigcup_{i\in[6k]}E_{A_i}$, $E_B=\bigcup_{i\in[6k]}E_{B_i}$ and $E=E_A \cup E_B$. It follows from (1) and (P4) that
\begin{equation}\label{2}
\begin{aligned}
  |E_A|&\leq 6k\cdot (\frac{\hat{\delta}^-(D)}{ 18000k^2}+20kl)\leq \frac{\hat{\delta}^-(D)}{3\cdot10^3k}+120k^2l \leq \frac{\hat{\delta}^-(D)}{ 2\cdot 10^3k},\\
   &|E_B|\leq \frac{\hat{\delta}^+(D)}{ 2\cdot 10^3k} \text{ and } |E|\leq |E_A|+|E_B|\leq \frac{\hat{\delta}^+(D)}{ 1000k}.
  \end{aligned}
  \end{equation}

Next we will construct the sets $V_1,V_2$ in several steps. We will color the vertices of $V(D)$ with colors $I$ and $II$. At
each step, we use $V_{I}$ and $V_{II}$ to denote the set of vertices of colors $I$ and $II$, respectively. We start with no vertices of $D$ colored, and we now begin to color all vertices in $D_0$. In what follows, let the vertices in
\begin{equation*}
\begin{aligned}
D_0^\prime = &\bigcup\limits_ {i\in[k]} (A_i\cup B_i)\cup   \bigcup\limits_{i\in[2k+1,3k]}
\{a_i,b_i\}\cup \bigcup\limits_{i\in[3k+1,4k]}
(A_i\cup B_i\setminus \{a_i,b_i\})\\ &\cup \bigcup\limits_{i\in[4k+1,5k]}
(A_i\cup \{b_i\}\setminus \{a_i\}) \cup \bigcup\limits_{i\in[5k+1,6k]}
(B_i\cup \{a_i\}\setminus \{b_i\})
\end{aligned}
\end{equation*}
  be colored with $I$ and all vertices in $D_0\setminus D_0^\prime$ be colored with $II$ (see Fig. \ref{fig1}).
\begin{figure}[H]
\centering    %居中       %子图居中
    \includegraphics[scale=0.65]{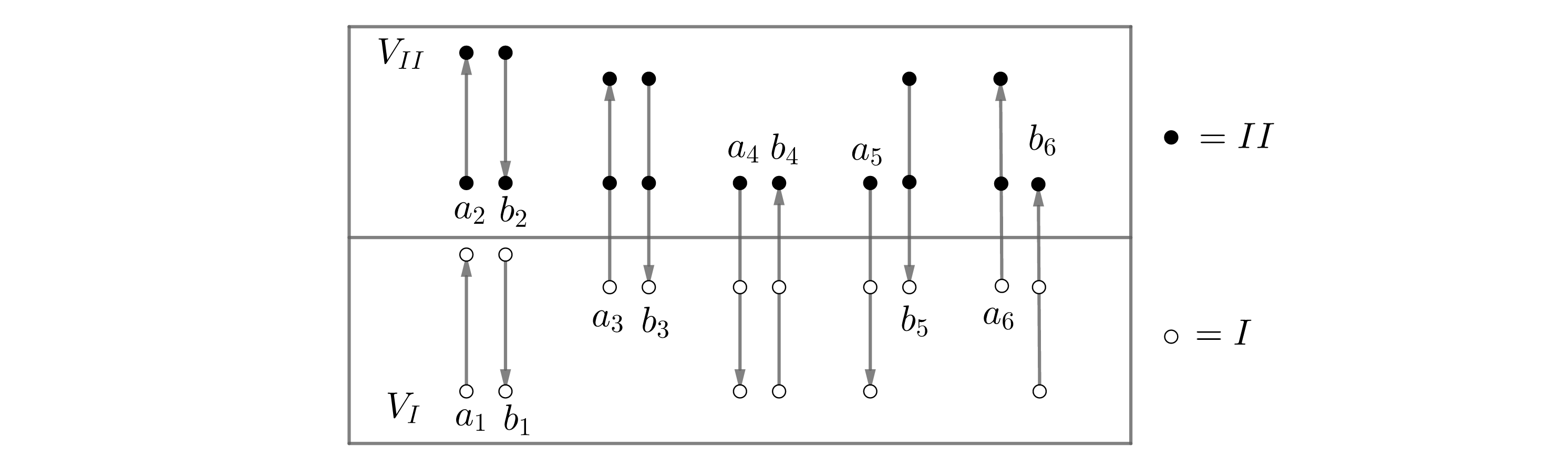}   %以pic.jpg 的0.5倍大小输出
\caption{ Color all vertices in $D_0$ in the case when $k = 1$.} %  %大图名称
\label{fig1}  %图片引用标记
\end{figure}

To solve the problem discussed earlier we will introduce the notion of ``safe" vertices.

 \begin{definition}\label{definition1}
(Safe vertex) For each $i\in \{I, II\}$, we call a vertex $v$ in $V_i$ safe if

\noindent (s1) $(v,V(D)\setminus (D_0\cup E_B))$ is $k$-connected in $D [V_i]$,

\noindent (s2) $(V(D)\setminus (D_0\cup E_A), v)$ is $k$-connected in $D [V_i]$,

\noindent (s3) $(v,V(D)\setminus (D_0\cup E_B))$ is $k$-connected in $D [V_I,V_{II}]$, and

\noindent (s4) $(V(D)\setminus (D_0\cup E_A), v)$ is $k$-connected in $D [V_I,V_{II}]$.
\end{definition}

 The following useful property can be proved by the definition of safe, so we omit the proof.

\begin{property}\label{lemma4}
For each $\{i,j\}= \{I,II\}$, the following statements hold.

\noindent (i) All colored vertices in $V(D)\setminus (D_0\cup E)$ are safe.

\noindent (ii) All colored vertices in $V(D)\setminus (D_0\cup E_B)$ satisfy (s1) and (s3); All  colored vertices in $V(D)\setminus (D_0\cup E_A)$ satisfy (s2) and (s4).

\noindent (iii) If $v\in V_i$ has at least $k$ out-neighbors of color $i$ satisfying (s1), then $v$
satisfies (s1); The analogue holds if $v$ has at least $k$ in-neighbors of color $i$ satisfying (s2).

\noindent (iv) If $v\in V_i$ has at least $k$ out-neighbors of color $j$ satisfying (s3), then $v$  satisfies (s3); The analogue holds if $v$ has at least $k$ in-neighbors of color $j$ satisfying (s4).
\end{property}

The following claim, which will be used frequently in the rest of the proof, states that to ensure the safety of the colored vertices at some step, we need to prepare a small set $C^\prime$ and a coloring of the vertices in $C^\prime$ such that all colored vertices are safe. In what follows, let $f(k,l)=k(k+l)\log(2kl)$.

\begin{claim}\label{claim1}
Let $l^\prime \in \mathbb{N}$ with $l^\prime\leq 8\cdot 10^3f(k,l)$. Let $C$ be a set of at most $3\cdot 10^4(k+l)f(k,l)$ vertices in $V(D)$  such that $X,Y\subseteq C$. Suppose that $W_I,W_{II}\subseteq V(D)\setminus C$ are the colored vertex sets of colors $I$ and $II$, respectively, and $|W_I|,|W_{II}|\leq l^\prime$. Then there is a set $C^\prime\subseteq V (D )\setminus (C \cup W_I\cup W_{II})$ and a coloring of the vertices
in $C^\prime$ such that every vertex in $C^\prime \cup W_I\cup W_{II}$ is safe and $|C^\prime|\leq 2kl^\prime+400f(k,l)$.
\end{claim}
\begin{proof}
To this end, we divide the proof into the following two steps. In Step 1, we color some vertices in $V (D)\setminus (C \cup W_I\cup W_{II})$ such that all colored vertices satisfy (s3) and (s4). In Step 2, we continue to color some vertices such that all colored vertices are safe.

\textbf{Step 1.} There exists a vertex set $C_1^\prime\subseteq V (D)\setminus (C \cup W_I\cup W_{II})$ of size at most
$$2k(2k+l+l^\prime-2)+100(k+l)\log(2kl)+200f(k,l)$$
and a coloring of the vertices
in $C_1^\prime$ such that all vertices in $C_1^\prime \cup W_I\cup W_{II}$ satisfy (s3) and (s4).

Let $D_1=D[W_I, W_{II}]$ and $l_0=l+l^\prime$. Since $|W_I|,|W_{II}|\leq l^\prime$ and $\delta(D)\geq n-l$, it follows that $\delta(D_1)\geq |D_1|-l_0$. Let $E^\prime= C\cup D_1 \cup D_0 \cup E$ and $E_A^\prime= C\cup D_1\cup D_0\cup E_A $. According to  (\ref{1}), (\ref{2}) and the upper bound of $|C|$, we obtain
\begin{equation}\label{3}
\begin{aligned}
\hat{\delta}^+(D)- |E^\prime|\geq  7\cdot 10^6(k+l)f(k,l) \text{ and } \hat{\delta}^-(D)- |E_A^\prime|\geq 7\cdot 10^6(k+l)f(k,l).
\end{aligned}
\end{equation}
By applying Lemma \ref{lemma3} to $D_1$, there are two sets $D_1^{t}, D_1^{h}\subseteq D_1$ such that
\begin{equation}\label{5}
|D_1^{t}|,|D_1^{h}|\leq 2k+l_0-2
\text{ and for  each } v\in D_1,  (v, D_1^{h}) \text{ and } (D_1^{t}, v) \text{ are } k\text{-connected in } D_1.
\end{equation}
To finish Step 1, we will find vertex sets $Q^\prime, Q, Q_1, Q_2$, and add vertex sets $Q^\prime, Q, Q_1, Q_2$ in turn into $D_1$ so that all colored vertices satisfy (s3) and (s4). So $C_1^\prime=Q^\prime\cup Q\cup  Q_1\cup Q_2$.

\textbf{Step 1.1.} Eq.(\ref{3}) implies that for each $u\in D_1^{h}$, we may greedily choose a set $Q^\prime_u \subseteq N^+_D(u)\setminus E^\prime$ with $|Q^\prime_u|=k$ and $Q^\prime_u\cap Q^\prime_{u^\prime}=\emptyset$ for any $u^\prime\in D_1^{h}\setminus \{u\}$. Then color the vertices in $Q^\prime_u$ the different color from $u$. Let $Q^\prime=\bigcup_{u\in D_1^{h}} Q^\prime_u$. This yields that for each $u\in D_1^{h}$, the pair $(u, Q^\prime)$ is $k$-connected in $D[V_I,V_{II}]$. By (\ref{5}), we have that $|Q^\prime|\leq k(2k+l_0-2)$ and that $(v, Q^\prime)$ is $k$-connected in $D[V_I,V_{II}]$ for each $v\in D_1$. Recall that $Q^\prime \cap (D_0\cup E)=\emptyset$, Property \ref{lemma4} (i) and (iii) yields that all vertices in $Q^\prime$ are safe and all vertices in $D_1$ satisfy (s3).

\textbf{Step 1.2.} Note that $|Q^\prime|\leq k(2k+l_0-2)$, in the same way as Step 1.1, we may also greedily choose a set $Q$ outsides $D_0\cup E_A^\prime \cup Q^\prime$ with $|Q|\leq k(2k+l_0-2)$ such that each vertex  in $D_1^{t}$ has $k$ different color in-neighbors in $Q$, and all colored vertices in $D_1\cup Q$ satisfy (s4).

\begin{figure}[H]
\centering    %居中       %子图居中
    \includegraphics[scale=0.55]{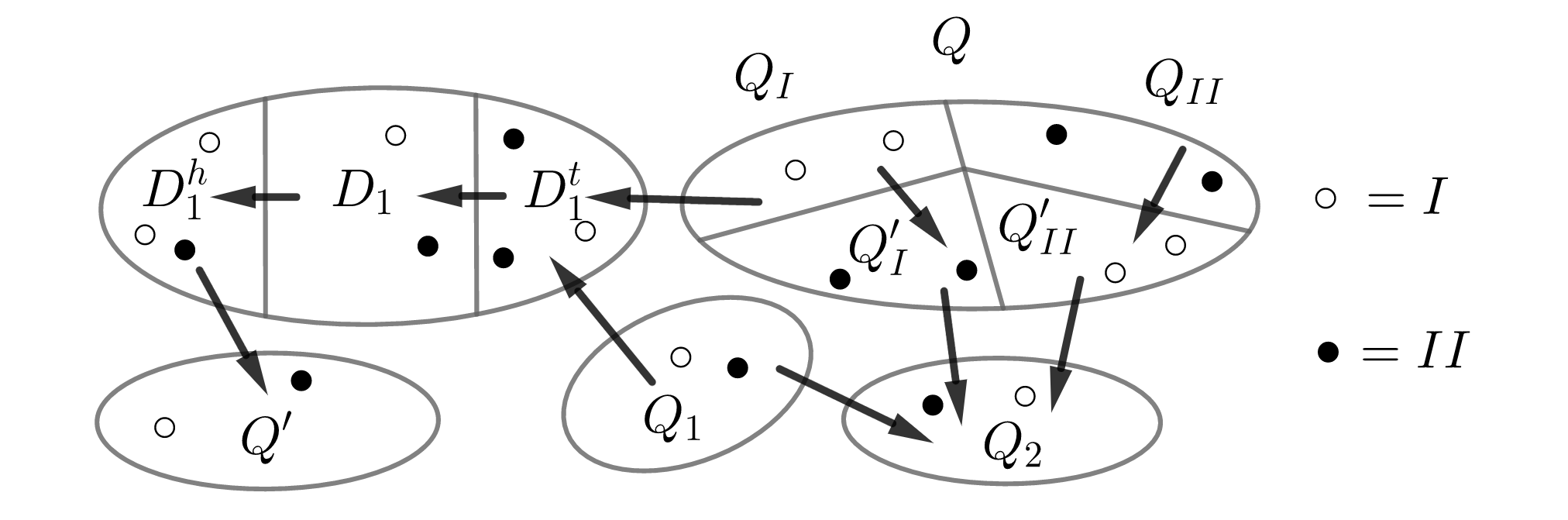}   %以pic.jpg的0.5倍大小输出
\caption{ The structure  of $D[W_I\cup W_{II}]$ at Step 1 in the proof of Claim \ref{claim1}.} %  % 大图名称
\label{fig2}  %图片引用标记
\end{figure}

\vspace{-3mm}
Combining this with Step 1.1, it remains to color some uncolored vertices to ensure that all colored vertices in $Q$ satisfy (s3). For each $i\in\{I,II\}$, let $Q_i$ denote the set of vertices colored $i$ in $Q$. Since $|Q_i|\leq |Q|\leq k(2k+l_0-2)$, applying Proposition \ref{proposition1} to $D[Q_i]$ to find a set $Q_i^\prime\subseteq  Q_i$ with
$|Q_i^\prime|\leq 3(k+l)\log |Q_i| \leq 50(k+l) \log(2kl)$ such that every vertex in $Q_i\setminus Q_i^\prime$ has at
least $k$ out-neighbors in $Q_i^\prime$. Then we change the color of the vertices in $Q_i^\prime$ to color $j$, where $\{i,j\}=\{I,II\}$ (see Fig. \ref{fig2}). This results in that some vertices in $D_1$ may not satisfy (s4), since some vertex $u\in D_1^t$ may have less than $k$ in-neighbors with different color from $u$ in $Q$.

However, according to (\ref{3}) and $|Q\cup Q^\prime|\leq 2k(2k+l_0-2)\leq 5\cdot 10^3(k+l)f(k,l)$, we can still greedily choose a set $Q_1\subseteq V(D)\setminus (Q^\prime\cup Q \cup E_A^\prime)$ with $|Q_1|\leq 100(k+l)
\log(2kl)$, and then color the vertices in $Q_1$ such that each vertex $u\in D_1^{t}$ has $k$ in-neighbors with different color from $u$. Since $Q_1\cap  E_A^\prime=\emptyset$, Property \ref{lemma4} (ii) and (iv) show that all vertices in $D_1\cup Q_1$ satisfy (s4). Therefore, it suffices to  guarantee that the vertices in $Q_1\cup Q$ satisfy (s3).  Owing to Property \ref{lemma4} (iv), if the vertices in $Q_I^\prime\cup Q_{II}^\prime$ satisfy (s3), then the vertices in $Q$ satisfy (s3). So we just need to ensure that the vertices in $Q_I^\prime\cup Q_{II}^\prime\cup Q_1$ satisfy (s3).

By (\ref{3}) and $|Q\cup Q^\prime\cup Q_1|\leq 2k(2k+l_0-2)+100(k+l)\log(2kl)\leq 5\cdot 10^3(k+l)f(k,l)$ again, we may greedily choose a set $Q_2 \subseteq V(D)\setminus (Q\cup Q^\prime\cup Q_1\cup E^\prime)$ with $|Q_2|\leq 200f(k,l)$, and then color the vertices in $Q_2$ such that each vertex $v\in Q_I^\prime\cup Q_{II}^\prime\cup Q_1$ has $k$ distinct out-neighbors with different color from $v$ in $Q_2$. Due to Property \ref{lemma4} (i), all vertices in $Q_2$ are safe, and then all vertices in $Q_I^\prime\cup Q_{II}^\prime\cup Q_1$ satisfy (s3) from Property \ref{lemma4} (iv).

Now all colored vertices in $V(D_1)\cup Q^\prime\cup Q\cup  Q_1\cup Q_2$ satisfy (s3) and (s4). Let $C_1^\prime=Q^\prime\cup Q\cup  Q_1\cup Q_2$. Then
$|C_1^\prime|\leq 2k(2k+l_0-2)+100(k+l)\log(2kl)+200f(k,l)$.

\textbf{Step 2.}  There exists a vertex set $C_2^\prime\subseteq V (D)\setminus (C \cup W_I\cup W_{II}\cup C_1^\prime)$ of size at most $68f(k,l)$ and a coloring of the vertices in $C_2^\prime$ such that all vertices in $W_I\cup W_{II}\cup C_1^\prime \cup C_2^\prime$ are safe.

Let $W_I^\prime$ and $W_{II}^\prime$ denote the set of vertices in $W_I\cup W_{II}\cup C_1^\prime$ of colors $I$ and $II$, respectively.  Write $E^{\prime\prime}= C\cup E \cup W_I^\prime \cup W_{II}^\prime \cup D_0$ and $E_A^{\prime\prime}= C\cup E_A \cup W_I^\prime \cup W_{II}^\prime \cup D_0$. By  (\ref{1}), (\ref{2}) and Step 1, we obtain
\begin{equation}\label{9}
\begin{aligned}
\hat{\delta}^+(D)- |E^{\prime\prime}|\geq  6\cdot 10^6(k+l)f(k,l) \text{ and } \hat{\delta}^-(D)- |E_A^{\prime\prime}|\geq 6\cdot 10^6(k+l)f(k,l).
\end{aligned}
\end{equation}

Next we will find $W, U^\prime, U^{\prime\prime}, W^\prime$, and add vertex sets $W, U^\prime, U^{\prime\prime}, W^\prime$ in turn to $W_I^\prime$ so that all vertices in $W_I^\prime \cup W\cup U^\prime \cup U^{\prime\prime} \cup W^\prime$ are safe. Similarly, we can also find some  vertices of order the same as $|W\cup U^\prime \cup U^{\prime\prime} \cup W^\prime|$ to ensure the safety of all vertices in $W_{II}^\prime$. This implies $|C_2^{\prime}|=2|W\cup U^\prime \cup U^{\prime\prime} \cup W^\prime|$.

%For each $j\in \{I,II\}$, in order to guarantee the safety of all vertices in $W_j^\prime$, we will find   vertex sets $W, U^\prime, U^{\prime\prime}, W^\prime$, and add vertex sets $W, U^\prime, U^{\prime\prime}, W^\prime$ in turn to $W_j^\prime$ so that all vertices in $W_j^\prime \cup W\cup U^\prime \cup U^{\prime\prime} \cup W^\prime$ are safe. This implies $|C_2^{\prime}|=2|W\cup U^\prime \cup U^{\prime\prime} \cup W^\prime|$.

Applying Lemma \ref{lemma3} to $D[W_I^\prime]$, there is a set $W_I^{t}\subseteq W_I^\prime$ such that $|W_I^{t}|\leq 2k+l-1$ (since $\delta(D)\geq n-l$) and for each $v\in W_I$, the pair $(W_I^{t}, v)$ is $k$-connected in $D[W_I^\prime]$. In a similar way as Step 1.2, we may choose a set $W$ with $|W|\leq k(2k+l-1)$ outside $E_A^{\prime\prime}$ such that each vertex $u\in W_I^{t}$ has $k$ in-neighbors with the same color as $u$ in $W$. Moreover, this together with Step 1 implies that
\begin{equation}\label{6}
\text{all vertices in } W \text{ satisfy } (s2) \text{ and } (s4), \text{ and all vertices in } W_I^\prime \text{ satisfy } (s2), (s3), (s4).
\end{equation}
\vspace{-10mm}
\begin{figure}[H]
\centering    %居中       %子图居中
    \includegraphics[scale=0.5]{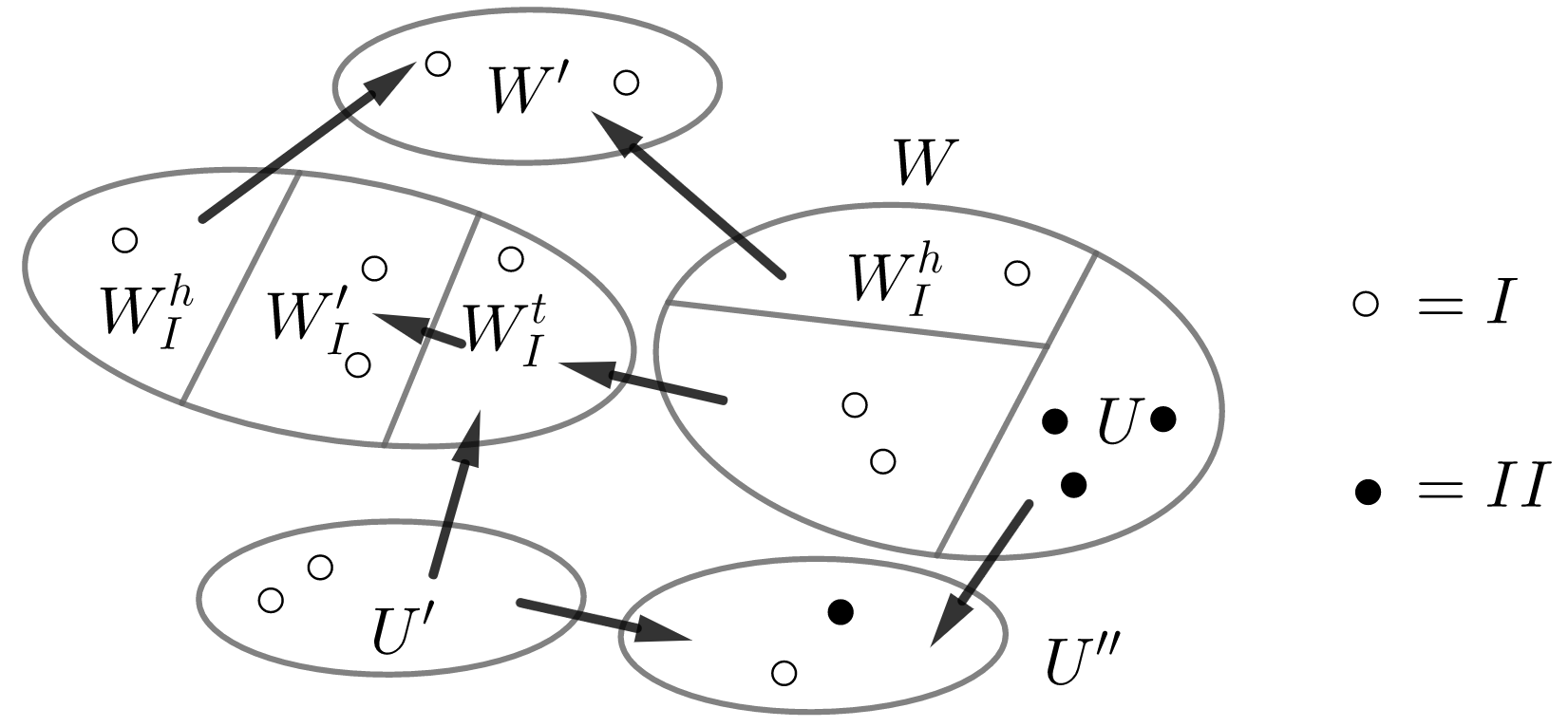}   %以pic.jpg 的0.5倍大小输出
\caption{ The structure  of $D[W_I]$ at Step 2 in the proof of Claim \ref{claim1}.} %  % 大图名称
\label{fig3} %图片引用标记
\end{figure}
\vspace{-3mm}

 Apply Proposition \ref{proposition1} to $D[W]$, there is a set $U\subseteq W$  with $|U|\leq 3(k+l) \log |W| \leq 6(k+l)\log(2kl)$ such that every vertex in $W\setminus U$ has at least $k$ out-neighbors in $U$. Now we change the color of the vertices in $U$ to $II$ (see Fig. \ref{fig3}). This operation results in that

\noindent (Q1) some vertices in $W_I^\prime$ may not satisfy (s2), since some vertex $u\in W_I^{t}$ may have less than $k$ in-neighbors with the same color as $u$ in $W$, and

\noindent (Q2) if all vertices in $U$ satisfy (s3), then all vertices in $W$ satisfy (s3) by Property \ref{lemma4} (iv).

In the same way as  Step 1.2, we can still greedily choose an uncolored set $U^\prime$ outside $E_A^\prime$ with $|U^\prime|\leq 6(k+l)\text{log}(2kl)$, and then color the vertices in $U^\prime$ by $I$ such that each vertex in $u\in W_I^{t}$ has at least $k$ in-neighbors with the same color as $u$. Property \ref{lemma4} (ii) and (iv) show that all vertices in $W_I^\prime$ satisfy (s2), and
\begin{equation}\label{11}
\text{all vertices in } U^\prime \text{ satisfy } (s2) \text{ and } (s4).
\end{equation}

Owing to (\ref{9}), we may also greedily choose an uncolored set $U^{\prime\prime}$ outside $E^\prime$ with $|U^{\prime\prime}|\leq 24f(k,l)$ such that each vertex $v\in U\cup U^\prime$ has $2k$ distinct out-neighbors  in $U^{\prime\prime}$, and color $k$ of them by $I$, and  $k$ of them by $II$. It follows from Property \ref{lemma4} (i), (iii) and (iv) that all vertices in $U^{\prime\prime}$ are safe, and all vertices in $U\cup U^\prime$  satisfy (s1) and (s3). By (\ref{6}), (\ref{11}) and (Q2), we get that all vertices in $U\cup U^\prime$ are safe, and all vertices in $W$ satisfy (s3). Therefore, it suffices to  guarantee that the vertices in $W\cup W_I^\prime\setminus U$ satisfy (s1).

Applying Lemma \ref{lemma3} to $D[W\cup W_I^\prime\setminus U]$, there is a set $W_I^h\subseteq W\cup W_I^\prime\setminus U$ such that $|W_I^h|\leq 2k+l-1$, and for each  $u\in W\cup W_I^\prime\setminus U$, the pair $(u, W_I^h)$ is $k$-connected in $D[W\cup W_I^\prime\setminus U]$. By (\ref{9}) and $|W \cup U^\prime \cup U^{\prime\prime}|\leq 200(k+l)f(k,l)$, we may greedily choose a set $W^\prime$ outside $E^{\prime\prime}$ in a same way as Step 1.1 such that $|W^\prime|\leq k(2k+l-1)$ and each vertex $u\in W_I^h$ has $k$ distinct out-neighbors with the same color as $u$ in $W^\prime$. Moreover, all vertices in $W^\prime$ are safe and all vertices in $W\cup W_I^\prime\setminus U$ satisfy (s1). Now all colored vertices in $W_I^\prime\cup W\cup U^\prime \cup U^{\prime\prime} \cup W^\prime$ are safe. So $|C_2^{\prime}|\leq 2\cdot (2k(2k+l-1)+6(k+l)\log(2kl)+ 24f(k,l))\leq 68f(k,l)$.

Set $C^\prime= C_1^{\prime} \cup C_2^{\prime}$. Then
\begin{equation*}
\begin{aligned}
|C^\prime|&\leq 2k(2k+l_0-2)+100(k+l)\log(2kl)+200f(k,l) +68f(k,l)\\
&\leq 2k(l+l^\prime)+368f(k,l)\\
&\leq 2kl^\prime+400f(k,l).
\end{aligned}
\end{equation*}
\end{proof}

\begin{claim}\label{claim2}
We can color some vertices outside $D_0$ such that every colored vertex is safe, and the set $C_1$ consisting of all vertices colored so far satisfies $|C_1|\leq 10^3(k+l)f(k,l)$.
\end{claim}
\begin{proof}
To prove Claim \ref{claim2}, for every $v\in X\cup Y$ we in turn greedily
choose $2k$ uncolored in-neighbors and $2k$ uncolored out-neighbors, all distinct from each other, and color $k$ of the in-neighbors and $k$ of the out-neighbors by $I$ and the remaining $2k$ in/out-neighbors by $II$. Let $D^\prime$ denote the new colored vertices. Due to Property \ref{lemma4} (iii) and (iv), if all vertices in $D^\prime$ are safe, then all vertices in $X\cup Y$ are also safe.

Let $W_I$ and $W_{II}$ be the sets of the vertices of colors $I$ and $II$ in $D_0\cup D^\prime \setminus (X\cup Y)$, respectively. Then $|W_I|,|W_{II}|\leq 6ck+24k^2$. By applying Claim \ref{claim1} with $C=X\cup Y$ and $l^\prime=6ck+24k^2$, we obtain
\begin{equation*}
\begin{aligned}
|C_1|&\leq 12ck+48k^2+2k\cdot(6ck+24k^2)+400f(k,l)\\
&\leq 10^3(k+l)f(k,l).
\end{aligned}
\end{equation*}
This completes the proof of Claim \ref{claim2}.
\end{proof}

\subsection{Construction of the correct $i$-path}
The aim of this subsection is to find the correct $i$-path for all $i\in [6k]$, and make sure that all vertices in these paths are safe. So, let us first give  the definition of the ``correct" $i$-path.

 \begin{definition}\label{definition2}
(Correct $i$-path) For each $i\in[6k]$, an i-path $P_i$ is a path from the head $b_i$ of $B_i$ to the tail $a_i$ of $A_i$, and an $i$-path $P_i$ is called correct if $i\in[1,2k]$ or it is odd  when $i\in[2k+1,4k]$ or it is even when $i\in[4k+1,6k]$.
\end{definition}

 The paths $P_i$ that we construct will be either ``short" or ``long". A path is said to be \emph{short} if its length is at most $1200(k+l)\log(2kl)+3l$. Otherwise we say it \emph{long}. Firstly, we will try to find the short disjoint correct $i$-path $P_i$ for each $i\in[6k]$. Let $\mathcal{P}_{s}$ be the set of short disjoint correct $i$-paths satisfying the following properties.

(O1) Int($\mathcal{P}_{s}) \subseteq V(D)\setminus C_1$.

(O2) For each $i \in [6k]$, $\mathcal{P}_{s}$ contains at most one $i$-path.

(O3) The number of paths in $\mathcal{P}_{s}$ is as large as possible.

 Let $I_{s}$ be the set of all indices $i$ for which we have been able to choose a correct short $i$-path and  $L=[6k]\setminus I_{s}$.

\begin{claim}\label{claim3}
 We can color all vertices in $V(\mathcal{P}_s)$ and some other vertices of $D$ such that

\noindent (i) every colored vertex is safe, and the set $C_2$ consisting of all vertices colored so far satisfies $|C_2|\leq 25000(k+l)f(k,l)$,

\noindent (ii) for each $i\in L$, any $i$-path whose internal vertices lie in $V(D)\setminus C_2$ has length at least $1200(k+l)\log(2kl)+2$.
\end{claim}
\begin{proof}
To show this claim, we begin by proving the following statements hold.

$(a)$ For each $i\in L$, there exist at most $l+1$ internally disjoint $i$-paths of length at most $2l + 2$ in $V(D)\setminus (V(\mathcal{P}_{s})\cup C_1)$.

\emph{Proof of (a).} Suppose not, then there exists an index $i_1\in L$ such that $V(D)\setminus (V(\mathcal{P}_{s})\cup C_1)$ has at least $l+2$ internally disjoint $i_1$-paths of length at most $2l+2$. Let $P_{i_1,1},\dots, P_{i_1,t}$ be these $i_1$-paths, where $t\geq l+2$, and  $P_{i_1,j}=b_{i_1}u_{i_1,j}^1\cdots u_{i_1,j}^{|P_{i_1,j}|-2} a_{i_1}$ with $j\in [t]$. According to (O3), these $i_1$-paths are incorrect. This implies that $u_{i_1,t}^1$ is not adjacent to $u_{i_1,j}^1$ for each $j\in [t-1]$, which contradicts that $\delta(D)\geq n-l$. Hence (a) holds.

For all $i\in [L]$, let $\mathcal{P}_s^{\prime}$ be a collection of $i$-paths in $V(D)\setminus (V(\mathcal{P}_{s})\cup C_1)$ of length at most $2l + 2$. Then $|\text{Int}(\mathcal{P}_s^{\prime})|\leq 16kl^2$.

$(b)$ For each $i\in L$, there exists at most one $i$-path whose internal vertices lie in $V(D)\setminus (V(\mathcal{P}_{s}\cup \mathcal{P}^\prime_{s}))$ and whose length is at most $1200(k+l)\text{log}(2kl)+1$ such that all these paths are disjoint.

\emph{Proof of (b).} Suppose not, then let $P_i = b_iv_1\cdots v_aa_i$ and $P_i^\prime = b_iu_1\cdots u_ba_i$ be such two internally disjoint $i$-paths for some $i\in L$. Owing to (O3) and (a), those paths must be incorrect and the length of those paths is at least $2l+3$. Since $\delta(D)\geq n-l$, it follows that there exists a vertex $u_{i_1}$ such that $i_1\in [2l]$ is odd, and either $v_1u_{i_1}\in A(D)$ or $u_{i_1}v_1\in A(D)$. We may assume that $v_1u_{i_1}\in A(D)$, and so $b_iv_1u_{i_1}u_{i_1+1}\cdots u_ba_i$ is a correct $i$-path which is disjoint from all the other paths in $\mathcal{P}_{s}$, and its length is at most $1200(k+l)\text{log}(2kl)+1+2l\leq 1200(k+l)\text{log}(2kl)+3l$, contrary to (O3). Hence (b) holds.

For all $i\in L$, let $\mathcal{P}_s^{\prime\prime}$ be a collection of such incorrect $i$-paths in (b). Thus the number of paths in $\mathcal{P}_s\cup \mathcal{P}_s^{\prime\prime}$ is at most $6k$, and then $|V(\mathcal{P}_s\cup \mathcal{P}_s^{\prime\prime})|\leq 7300f(k,l)$.

Now we color all vertices in $V(\mathcal{P}_s\cup \mathcal{P}_s^{\prime}\cup \mathcal{P}_s^{\prime\prime})$, and color all vertices in $\text{Int}(\mathcal{P}_s^{\prime} \cup\mathcal{P}_s^{\prime\prime})$ by $I$. If $i\in[2k]\cap I_s$, then we color the vertices of $P_i$ with the same color. If $i\in[2k+1,6k]\cap I_s$, then we color the vertices on $P_i$ such that the color  alternates along the orientation of $P_i$. Together with Claim \ref{claim1} (applied with $C=C_1$ and $l^\prime=7300f(k,l)+16kl^2$) and Claim \ref{claim2}, this implies that every colored vertex is safe, and $|C_2|\leq |C_1|+ 7300f(k,l)+16kl^2 +2k\cdot (7300f(k,l)+16kl^2) +400f(k,l)
\leq 25000(k+l)f(k,l)$.
 \end{proof}

Set $D^\prime=(D\setminus C_2)\cup  \bigcup_{i\in L}(a_i\cup b_i)$.  Recall that $D$ is a strongly $10^7(k+l)f(k,l)$-connected digraph, Claim \ref{claim3} (i) implies that $D^\prime$ is strongly $96(k+l) \cdot 10^5 f(k,l)$-connected. Theorem \ref{corollary1} implies that $D^\prime$ is $192000 kf(k,l)$-linked. Together with Proposition \ref{proposition2} (i) and Claim \ref{claim3} (ii), this implies that for each $i\in L$ we can find $32000 f(k,l)$ $i$-paths in $D^\prime$ such that all these
$32000f(k,l)|L|$ paths whose internal vertices avoid $C_2$ have length at
least $1200(k+l)\log(2kl)+2$, and are internally disjoint.

 Proposition \ref{proposition2} (ii) yields that  for all $i\in L$ there exists $800f(k,l)$ internally disjoint $i$-paths such that the number of all vertices in these $i$-paths is at most $(n-|C_2|)/40$.  We choose this collection of $800f(k,l)|L|$ paths
such that the length of these paths is minimal. Set $L_0=[800 f(k,l)]$.

% for all $i\in L$ there exists $5l$ internally disjoint $i$-paths, say $P_{i,1},\ldots,P_{i,5l}$, such that

For all $i\in L$ and all $j \in L_0$, let $P_{i,j}$ denote the $j$th $i$-path we have chosen. Then
\begin{equation}\label{16}
|\bigcup\limits_{(i,j)\in L\times L_0}\text{Int}(P_{i,j})|\leq \frac{n}{40}.
\end{equation}
Let $P^1_{ i,j}, P^2_{ i,j}, P^3_{ i,j}$  be
disjoint segments of the interior of $P_{i,j}$ so that $P_{i,j} = b_iP^1_{i,j}P^2_{i,j} P^3_{i,j}a_i$, and
\begin{equation}\label{13}
|P^1_{i,j}|=|P^3_{i,j}|=600(k+l)\log(2kl).
\end{equation}
Write
\begin{equation*}
\begin{aligned}
P^\alpha=\bigcup\limits_{(i,j)\in L\times L_0 } V(P^\alpha_{ i,j}) \text{ with } \alpha\in [3],\text{ and }
P_{i,j}=b_iv_{i,j}^1\cdots v_{i,j}^{|P_{i,j}|-2}a_i.
\end{aligned}
\end{equation*}
Owing to (\ref{13}), we obtain
\begin{equation}\label{15}
|P^1|=|P^3|\leq 600(k+l)\log(2kl)\cdot 6k \cdot 800 f(k,l) \leq 4\cdot 10^6(2kl)^6\leq 2^{22}(2kl)^6.
\end{equation}
Further, for each $(i,j)\in L \times L_0$, the fact that $\delta(D)\geq n-l$ and the minimality of these paths implies the following statements, where $\overline{N}(v)=\{u\in V(D): u \text{ is not adjacent to } v\}$.

(A1) Every vertex  $v_{i,j}^t\in V(P_{i,j})$ dominates $V(b_iP_{i,j}v_{i,j}^{t-2})\setminus \overline{N}(v_{i,j}^t)$  and is dominated by $V(v_{i,j}^{t+2}P_{i,j}a_i)\setminus \overline{N}(v_{i,j}^t)$.

(A2) If a vertex $u\notin C_2\cup P^1 \cup P^2 \cup P^3$ dominates $v_{i,j}^t\in V(P_{ i,j})$ (resp. is dominated by $v_{i,j}^t\in V(P_{ i,j})$), then $u$ dominates $V(b_iP_{i,j}v_{i,j}^{t-3})\setminus \overline{N}(u)$ (resp. is dominated by $V(v_{i,j}^{t+3}P_{i,j} a_i)\setminus \overline{N}(u)$).

Notice that Proposition \ref{proposition2} does not provide any information on the length of $P^2_{i,j}$ for each $(i,j)\in L\times L_0$, so we cannot apply Claim \ref{claim1}  to find a small set to guarantee the safety of all vertices in $V(P^2_{ i,j})$. However, if there exist two subsets $W^1_{i,j}\subseteq V(P^1_{i,j})$ and $W^3_{i,j}\subseteq V(P^3_{ i,j})$ with  $|W^1_{i,j}|=|W^3_{i,j}|=2(k+l+2)$ for each $(i,j)\in L\times L_0$  and a coloring of the vertices in $W_{i,j}^1\cup W_{i,j}^2$
such that there are $k+l+2$ vertices colored by $I$, and $k+l+2$ vertices colored by $II$ in each of $W_{i,j}^1$ and $W_{i,j}^2$, and each vertex of $W^1_{i,j}\cup W^3_{i,j}$ is safe, then (A1) and Property \ref{lemma4} (iii)-(iv) imply that all vertices in $\bigcup_{(i,j)\in L\times L_0}V(P^2_{ i,j})$ will be safe with whichever color. Moreover, by (A2) and Property \ref{lemma4} (iii)-(iv), for an uncolored vertex $v\in  E\setminus (C_2\cup P^1 \cup P^2 \cup P^3)$, if  it has an out-neighbor in  $V(P_{i,j}^2)$, then $v$ will satisfy (s1) and (s3) with whichever color; If it has an in-neighbor in $V(P_{i,j}^2)$, then $v$ will satisfy (s2) and (s4) with whichever color. This means that we can  guarantee the safety of a few vertices to ensure the safety of more. Therefore, our aim is to find such two disjoint subsets for each $(i,j)\in L\times L_0$,  whose safety can be guaranteed by two disjoint smaller subsets $P_I,P_{II}\subseteq P^1 \cup P^3$ (see Claims \ref{claim4} and \ref{claim7}).

%In practice, we will find two disjoint smaller subsets $P_I,P_{II}\subseteq P^1 \cup P^3$ (see Claim \ref{claim4}) and a coloring of all vertices in $P_I\cup P_{II}$ such that all vertices in $P_I\cup P_{II}$  are safe  to guarantee the safety of $W^1_{i,j}\cup W^3_{i,j}$ for each $(i,j)\in L\times [800 k(k+l)\text{log}(2kl)]$ (see Claim \ref{claim7}).

%Further, there still exist at least $40 k(k+l)\text{log}(2kl)$ $i$-paths for each $i\in L$. Finally, we will use $5l$ $i$-paths to construct one correct $i$-path  for each $i\in L$ (see Claim \ref{claim8}), and we color them such that each vertex of these paths is safe (see Claim \ref{claim10}).

\begin{claim}\label{claim4}
There exist two disjoint sets  $P_I,P_{II}\subseteq P^1 \cup P^3$ with
$|P_I|,|P_{II}|\leq  90(k+l)\text{log}(2kl)$ such that every vertex in $P^1 \cup P^3\setminus (P_I\cup P_{II})$ has at least $k$ out-neighbors and at least $k$
in-neighbors in each of $P_{I}$ and $P_{II}$.
\end{claim}
\begin{proof}
To prove this claim, apply Proposition \ref{proposition1} and (\ref{15}) to $D[P^1 \cup P^3]$ to find a set $P_I \subseteq  P^1 \cup P^3$ with
$|P_I|\leq 3(k+l) \log(2^{23}(2kl)^6) \leq  90(k+l)\log(2kl)$ such that every vertex in $P^1 \cup P^3\setminus P_{I}$ has at
least $k$ out-neighbors and at least $k$ in-neighbors in $P_{I}$. By applying
Proposition \ref{proposition1} and (\ref{15}) to $D[P^1 \cup P^3\setminus P_{I}]$, there exists a set $P_{II} \subseteq P^1 \cup P^3\setminus P_{I}$ with $|P_{II}|\leq 90(k+l)\log(2kl)$ such that every vertex in $P^1 \cup P^3\setminus (P_I\cup P_{II})$ has at least $k$ out-neighbors and  at least $k$ in-neighbors in $P_{II}$. This completes the proof of Claim \ref{claim4}.
\end{proof}

Let $L_1 = \{(i,j) : V(P_{i,j}) \cap (P_I\cup P_{II})\}\neq \emptyset$. For all  $(i,j)\in L_1$ and $\alpha \in \{1,3\}$, define
\begin{center}
$U^\alpha_{i,j}$ to be a subset of $V(P^\alpha_{i,j})\setminus (P_I\cup P_{II})$ with $|U^\alpha_{i,j}|=2(k+l+2)$.
\end{center}
 This is possible, since $|V(P^\alpha_{i,j})\setminus (P_I\cup P_{II})| \geq 400(k+l)\text{log}(2kl)$. Further, set $U=\bigcup_{(i,j)\in L_1}U^1_{i,j} \cup U^3_{i,j}$. It follows from Claim \ref{claim4} that
\begin{equation}\label{12}
 |U|\leq 180(k+l)\log(2kl)\cdot 4(k+l+2) \leq 2\cdot 10^3(k+l)^2\log(2kl).
\end{equation}

\begin{claim}\label{claim7}
We may color some uncolor vertices lying outside $U$ such that

\noindent (i) all vertices in $P_I$ are colored I and all vertices in $P_{II}$ are colored II,

\noindent (ii) all colored vertices are safe, the set $C_3^\prime$ consisting of the colored vertices at this step satisfies $|C_3^\prime\setminus(P_{I}\cup P_{II})| \leq 580f(k,l)$, and the set $C_3$ consisting of all vertices colored so far satisfies $|C_3|\leq 3\cdot 10^4(k+l)f(k,l)$,

\noindent (iii) all vertices in $P^1 \cup P^3 \setminus (P_{I}\cup P_{II})$  will be safe with whichever color.
\end{claim}
\begin{proof}
Color the vertices in $P_I\cup P_{II} $ as (i) required. According to Claim \ref{claim3} and (\ref{12}), we have $|C_2\cup  U|\leq 3\cdot 10^4(k+l)f(k,l)$. Apply Claim \ref{claim1} with $C=C_2\cup  U$ and $l^\prime= 90(k+l)\log(2kl)$ to obtain a set $C^\prime\subseteq V(D)\setminus (C_2\cup U \cup P_I\cup P_{II})$ such that every vertex in $P_I\cup P_{II} \cup C^\prime$ is safe, and $|C_3^\prime\setminus(P_{I}\cup P_{II})|=|C^\prime|
\leq 180f(k,l) +400f(k,l)\leq580f(k,l)$.
 Combining this with Claims \ref{claim3} and \ref{claim4} yields
$|C_3|\leq |C_2|+ |C^\prime| +|P_I\cup P_{II}| \leq 3\cdot 10^4(k+l)f(k,l)$, and then (iii) follows from (i), (ii), Claim \ref{claim4} and Property \ref{lemma4} (iii)-(iv).
\end{proof}

Let $L_2 = \{(i,j) : V(P_{i,j}) \cap (C_3^\prime\setminus (P_I\cup P_{II}))\}\neq \emptyset$. For all  $(i,j)\in L_2$ and $\alpha \in \{1,3\}$, define
\begin{center}
$V^\alpha_{i,j}$ to be a subset of $V(P^\alpha_{i,j})\setminus (C_3^\prime\setminus (P_I\cup P_{II}))$ with $|V^\alpha_{i,j}|=2(k+l+2)$.
\end{center}
This is also possible, since $|V(P^\alpha_{i,j})\setminus (C_3^\prime\setminus (P_I\cup P_{II}))| \geq 20f(k,l)$. Set $V=\bigcup_{(i,j)\in L_2}V^1_{i,j} \cup V^3_{i,j}$. It follows from Claim \ref{claim7} (ii) that
\begin{equation}\label{14}
 |V|\leq 580f(k,l)\cdot 4(k+l+2) \leq 5\cdot 10^3(k+l)f(k,l).
\end{equation}

\begin{claim}\label{claim5}
For  $(i,j)\in L_1\cup L_2$ and $\alpha \in \{1,3\}$ we may color all vertices in $U_{i,j}^\alpha \cup V_{i,j}^\alpha$ such that

\noindent (i) there are $k+l+2$ vertices colored by I, and $k+l+2$ vertices colored by II in each of $U_{i,j}^\alpha$ and $V_{i,j}^\alpha$ (see Fig. \ref{fig5}),

\noindent (ii) all colored vertices are safe, and the set $C_4$ consisting of all vertices colored so far satisfies $|C_4|\leq 4\cdot 10^4(k+l)f(k,l)$,

\noindent (iii) all vertices in $V(P^2_{i,j})\setminus (U \cup V)$  will be safe with whichever color.
\end{claim}
\begin{proof}
Recall that $|U^\alpha_{i,j}|=|V^\alpha_{i,j}|=2(k+l+2)$, so we can color the vertices in $U_{i,j}^\alpha \cup V_{i,j}^\alpha$ such that (i) holds (see Fig. \ref{fig5}). Moreover, $U^\alpha_{i,j} \cup V^\alpha_{i,j} \subseteq P^1 \cup P^3\setminus (P_I \cup P_{II})$, this together with Claim \ref{claim7} (iii), (\ref{12}) and (\ref{14}) implies that all colored vertices in $U_{i,j}^\alpha \cup V_{i,j}^\alpha$ are safe, and $|C_4|\leq 4\cdot 10^4(k+l)f(k,l)$. Hence (ii) holds, and then (iii) follows from (i), (ii), (A1) and Property \ref{lemma4} (iii)-(iv). This completes the proof of Claim \ref{claim5}.
\end{proof}
\vspace{-5mm}
\begin{figure}[H]
\centering    %居中       %子图居中
    \includegraphics[scale=0.65]{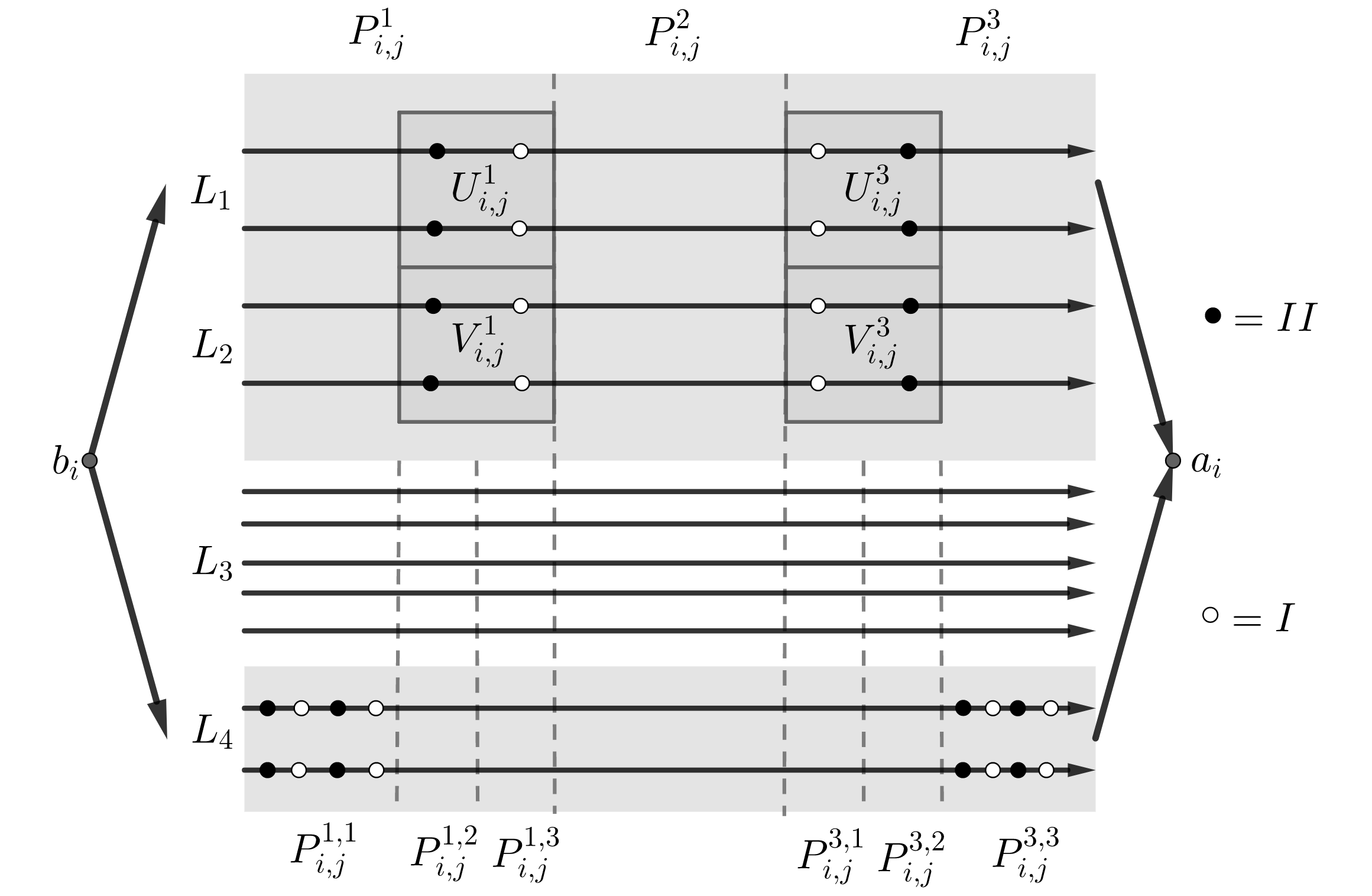}   %以pic.jpg 的0.5倍大小输出
\caption{ The structure of $P_{i,j}$ with $(i,j)\in L\times L_0$.} %  % 大图名称
\label{fig5}  %图片引用标记
\end{figure}

Recall that for each $i\in L$ there are $800 f(k,l)$ internally disjoint $i$-paths, Claims \ref{claim4} and \ref{claim7} (ii) state that $|C_3^\prime| \leq 760f(k,l)$. This implies that for all $i\in L$ there exist at least $40f(k,l)$ internally disjoint $i$-paths whose internal vertices avoid $C_4$. Next, for all $i\in L$ we will use $5l$ such $i$-paths to construct one correct $i$-path; assume they are $P_{i,1},\dots, P_{i,5l}$.

For all $(i,j)\in L\times L_0 \setminus(L_1 \cup L_2)\}$ and all $\alpha\in \{1,3\}$, let $P^{\alpha,1}_{ i,j}, P^{\alpha,2}_{ i,j}, P^{\alpha,3}_{ i,j}$  be
disjoint segments of $P^\alpha_{i,j}$ so that $P^\alpha_{i,j} = P^{\alpha,1}_{i,j}P^{\alpha,2}_{i,j}P^{\alpha,3}_{i,j}$ (see Fig. \ref{fig5}), and
$$|P^{\alpha,1}_{i,j}|=|P^{\alpha,3}_{i,j}|=2(k+l+2).$$
By (\ref{13}), we obtain
\begin{equation}\label{17}
 |P^{\alpha,2}_{i,j}|\geq 4(k+l+1).
\end{equation}
Further, let
\begin{equation*}
  \begin{aligned}
  L_3=\{(i,j):(i,j)\in L\times \{1,\ldots,5l\}\} \text{ and }
  L_4=\{(i,j):(i,j)\in L\times L_0 \setminus(L_1 \cup L_2 \cup L_3)\}.
  \end{aligned}
\end{equation*}

Analysis similar to that in the proof of Claim \ref{claim5} shows the following statements hold.
\begin{claim}\label{claim9}
For all $(i,j)\in L_4$ we may color all vertices in $V(P_{i,j}^{1,1} \cup P_{i,j}^{3,3})$ such that

\noindent (i) the color of the vertices on $P_{i,j}^{1,1}$ (resp. $P_{i,j}^{3,3}$) alternates along the orientation of $P_{i,j}^{1,1}$ (resp. $P_{i,j}^{3,3}$) (see Fig. \ref{fig5}), and all colored vertices are safe,

\noindent (ii) all vertices in $V(P_{i,j})\setminus V(P_{i,j}^{1,1} \cup P_{i,j}^{3,3})$  will be safe with whichever color.
\end{claim}

Recall that $P_{i,j}=b_iv_{i,j}^1\cdots v_{i,j}^{|P_{i,j}|-2}a_i$,  for each $i\in [L]$ consider the subdigraph $D_i$ of $D$ induced by $v_{i,j}^{|P_{i,j}|-2k-2l-6}$ (the last vertex of $P^{3,2}_{i,j}$) for $j\in [5l]$. Eq.(\ref{17}) and the fact that $\delta(D)\geq n-l$ shows the following statements hold, where $j_1,j_2,j_3,j_4\in  [5l]$ and $\gamma$ is an even integer.

(B1) The subdigraph $D_i$ has at least two vertices of out-degree at least two; assume they are $v_{i,j_1}^{|P_{i,j_1}|-2k-2l-6}$ and $v_{i,j_2}^{|P_{i,j_2}|-2k-2l-6}$. Further, assume $v_{i,j_2}^{|P_{i,j_2}|-2k-2l-6}$ dominates $v_{i,j_3}^{|P_{i,j_3}|-2k-2l-6}$ and $v_{i,j_1}^{|P_{i,j_1}|-2k-2l-6}$ dominates $v_{i,j_4}^{|P_{i,j_4}|-2k-2l-6}$, where $j_4\neq j_2$ and $j_3\neq j_1$.

(B2) $v_{i,j_1}^{2k+2l+5}$ dominates  $v_{i,j_2}^{2k+2l+4+\gamma}$, or  $v_{i,j_2}^{2k+2l+4+\gamma}$  dominates $v_{i,j_1}^{2k+2l+5}$; assume $v_{i,j_1}^{2k+2l+5}$ dominates $v_{i,j_2}^{2k+2l+6}$.

\begin{claim}\label{claim8}
For each $i\in L$ there is a correct i-path $P_i$ (see Fig. \ref{fig4}) such that

\noindent (i) all vertices in $\text{Int}(P_i)$ have not been colored, and $P_i$ and $P_{i^\prime}$ are disjoint whenever $i\neq i^\prime$,

\noindent (ii) if $i\in L\cap [2k]$, then $P_i = b_iP_{i,j_1}^{1,1}v_{i,j_1}^{2k+2l+5}v_{i,j_2}^{2k+2l+6}\cdots v_{i,j_2}^{|P_{i,j_2}|-2k-2l-6}v_{i,j_3}^{|P_{i,j_3}|-2k-2l-6} P_{i,j_3}^{3,3}a_i$ where $j_1,j_2,j_3\in[5l]$,

\noindent (iii) if $i\in L\cap [2k+1,6k]$, then either $P_i = P_{i,j}$ for some $j\in[5l]$, or $P_i =b_iP_{i,j_1}^{1,1}v_{i,j_1}^{2k+2l+5}v_{i,j_2}^{2k+2l+6}\cdots$ \\ $v_{i,j_2}^{|P_{i,j_2}|-2k-2l-6}v_{i,j_3}^{|P_{i,j_3}|-2k-2l-6} P_{i,j_3}^{3,3}a_i$ where $j_1,j_2,j_3\in[5l]$.
\end{claim}
\begin{proof}
Obviously, (i) and (ii) hold. Suppose that $i\in L\cap [2k+1,6k]$. If there exists a correct $i$-path among $P_{i,j}$ where $j\in [5l]$, then
we take it to be $P_i$. Otherwise, let $P_i$ be as described in (iii) (see Fig. \ref{fig4}). This completes the proof of Claim \ref{claim8}.
\end{proof}

\begin{claim}\label{claim10}
For all $(i,j)\in L_3$ we may color all vertices in $V(P_i \cup P^1_{i,j} \cup P^3_{i,j})$ such that

\noindent (i) all colored vertices are safe, and the set $C_5$ consisting of all vertices colored so far has size $|C_5|\leq 4\cdot 10^4(k+l)f(k,l)+n/40$ and
$$|C_5\setminus (\bigcup_{(i,j)\in L_3}V(P_{i,j}^2) \cup \bigcup_{(i,j)\in L_4}V(P_{i,j}^{1,1}\cup P_{i,j}^{3,3}))|\leq 8\cdot 10^4(k+l)f(k,l),$$

\noindent (ii) every vertex in $V(P_{i,j}^2)$ will be safe with whichever color, and

\noindent(iii) for an uncolored vertex $v\notin  P^1 \cup P^2 \cup P^3$, if it has an out-neighbor in  $V(P_{i,j}^2)$, then $v$ will satisfy  (s1) and (s3) with whichever color; If it has an in-neighbor in $V(P_{i,j}^2)$, then $v$ will satisfy  (s2) and (s4) with whichever color.
\end{claim}
\begin{proof}
First we color the vertices in $P_i$ for each $i\in L$. If $i\in L\cap [k]$, we color the vertices on $P_i$ by $I$; If $i\in L\cap [k+1,2k]$, we color the vertices on $P_i$ by $II$; If $i\in L\cap [2k+1,6k]$, we color the vertices on $P_i$ alternating along the orientation of $P_i$.

 For each $(i,j)\in L_3$, color all vertices in $V(P_{i,j}^{1,1}\cup P_{i,j}^{3,3})\setminus V(P_i)$ such that there are $k+l+2$ vertices colored by $I$, and $k+l+2$ vertices colored by $II$ in each of $V(P_{i,j}^{1,1})$ and $V(P_{i,j}^{3,3})$ (see Fig. \ref{fig4}). We color all vertices in $V(P_{i,j}^{1,2}\cup P_{i,j}^{3,2})\setminus V(P_i)$ by $I$, and color all vertices in $V(P_{i,j}^{1,3}\cup P_{i,j}^{3,1})\setminus V(P_i)$ by $II$ (see Fig. \ref{fig4}).

It follows from (A1) that each vertex in $V(P^2_{i,j})$ has at least $k$ out-neighbors in each color in $V(P_{i,j}^{1} )$, and at least $k$ in-neighbors in each color in $V(P_{i,j}^{3})$.  By Claim \ref{claim7} (iii), all colored vertices in $P^1 \cup P^2$ are safe. This together with Property \ref{lemma4} (iii)-(iv) yields that (ii) holds, and all colored vertices are safe. By (A2) and Property \ref{lemma4} (iii)-(iv), we obtain that (iii) holds.

 Owing to Claims \ref{claim7}--\ref{claim9}, all colored vertices belong to $C_3 \cup P^1 \cup P^2 \cup P^3$. It follows from (\ref{16}) that $|C_5|\leq |C_3\cup P^1 \cup P^2 \cup P^3|\leq 3\cdot 10^4(k+l)f(k,l)+n/40$. Further, we have
\begin{equation*}
  \begin{aligned}
  &|C_5\setminus (\bigcup_{(i,j)\in L_3}V(P_{i,j}^2) \cup \bigcup_{(i,j)\in L_4}V(P_{i,j}^{1,1}\cup P_{i,j}^{3,3}))|=|C_5|+|\bigcup_{(i,j)\in L_3}V(P_{i,j}^1\cup P_{i,j}^3)|\\
 \overset{(\ref{13})}{\leq}& 4\cdot 10^4(k+l)f(k,l)+6k\cdot 5l\cdot 1200(k + l)\log(2kl)\\
  \leq& 8\cdot 10^4(k+l)f(k,l).
  \end{aligned}
\end{equation*}
This completes the proof of Claim \ref{claim10}.
\end{proof}
\vspace{-3mm}
 \begin{figure}[H]
\centering    %居中       %子图居中
    \includegraphics[scale=0.9]{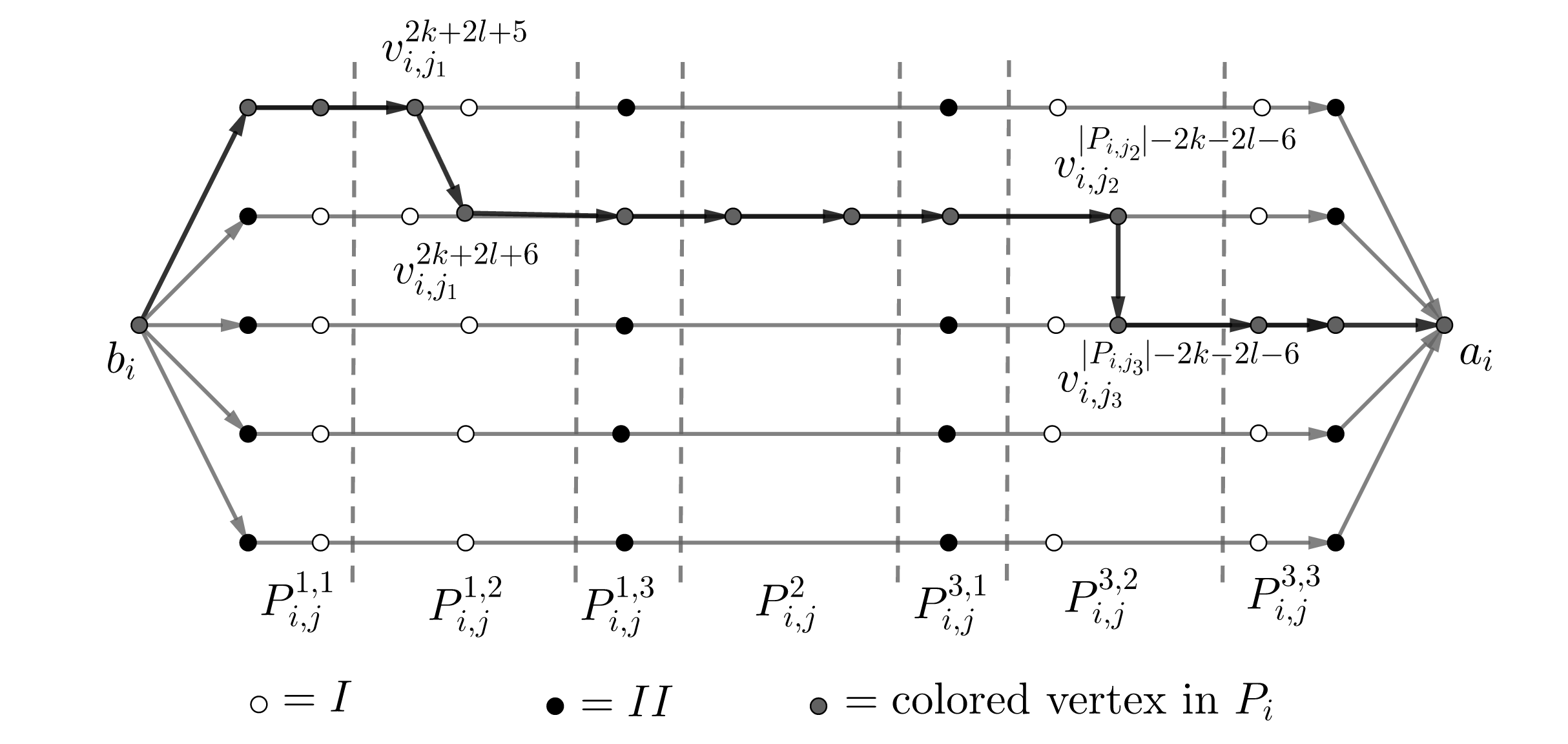}   %以pic.jpg 的0.5倍大小输出
\caption{Color patterns of the paths int($P_i$) with $i \in L\cap  [6k]$ and $l=1$. The black arrows indicate $P_i$.} %  % 大图名称
\label{fig4}  %图片引用标记
\end{figure}

\subsection{Coloring of the vertices in $E$}
In this subsection, we will distribute these vertices in exceptional sets into $V_1$ and $V_2$ so that these vertices are safe. Let
$W=\{w: w\in \bigcup_{(i,j)\in L\times L_0}V(P_{i,j})\setminus C_5\}$.
According to Claims \ref{claim7} (iii), \ref{claim5} (iii), \ref{claim9} (ii) and \ref{claim10} (ii), we obtain the following holds.
\begin{claim}\label{claim12}
All vertices in  $W$  will be safe with whichever color.
\end{claim}
Finally, it suffices to deal with the uncolored vertices in $E\setminus W$.

\begin{claim}\label{claim11}
We may color all vertices in $E\setminus W$ such that all colored vertices are safe, and the set $C_6$ consisting of all vertices colored so far satisfies $|C_6|\leq n/20$.
\end{claim}
\begin{proof}
We divide the proof of this claim into the following two steps. We first deal with all vertices in $E_A \setminus (C_5\cup W)$ (see Step 1) such that all the vertices in $E_A \setminus (C_5\cup W)$ satisfy (s2) and (s4). Property \ref{lemma4} (ii) yields that all vertices in $E_A\setminus E_B$ satisfy (s1) and (s3), and all vertices in $E_B\setminus E_A$  satisfy (s2) and (s4). After Step 1, all vertices in $E_A\setminus E_B$ are safe, and all vertices in $E_A\cap E_B$ satisfy (s2) and (s4). This implies that all vertices in $E_B$ satisfy (s2) and (s4). At Step 2, we color each vertex in $E_B \setminus (C_5\cup W)$ such that all the vertices in $E_B \setminus (C_5\cup W)$ satisfy (s1) and (s3). Then all vertices in $E \setminus (C_5\cup W)$ are safe.

\textbf{Step 1.} We can color each vertex in $E_A \setminus (C_5\cup W)$ in turn as well as some set $Z_A$ of additional vertices such that all the vertices in $E_A \setminus (C_5\cup W)$ satisfy (s2) and (s4). At first, let $Z_A =\emptyset$. At each  iteration, we
will update $Z_A$ such that

(a) $Z_A$ consists of colored vertices and $Z_A\cap (C_5 \cup E_A) = \emptyset$,

 (b) every colored vertex lies in $C_5 \cup E_A \cup  Z_A$,

 (c) $|Z_A| \leq  2k|E_A|$.

Now suppose that we have already colored some vertices in $E_A \setminus (C_5\cup W)$ such that $Z_A$ satisfies (a)-(c). It follows from (\ref{1}), (\ref{2}) and Claim \ref{claim10} (i) that
\begin{align}\label{7}
 &\hat{\delta}^-(D)-|E_A\cup Z_A|- |C_5\setminus (\bigcup_{(i,j)\in L_3}V(P_{i,j}^2) \cup \bigcup_{(i,j)\in L_4}V(P_{i,j}^{1,1}\cup P_{i,j}^{3,3}))|\nonumber \\
 \geq &\hat{\delta}^-(D) -3k|E_A| - |C_5\setminus (\bigcup_{(i,j)\in L_3}V(P_{i,j}^2) \cup \bigcup_{(i,j)\in L_4}V(P_{i,j}^{1,1}\cup P_{i,j}^{3,3}))|\nonumber\\
 \geq &2500(k+l)f(k,l)+ 2k.
\end{align}
 Since all colored vertices lie in $C_5\cup  E_A \cup Z_A$, eq.(\ref{7}) implies that for a vertex $v\in E_A \setminus (C_5\cup W)$, it satisfies one of the following statements.

(D1) It has at least an in-neighbor in $\bigcup_{(i,j)\in L_3}V(P_{i,j}^2)$.

 (D2) It has at least $2k$ uncolored in-neighbors $v_1, v_2, \ldots , v_{2k}$ outside $E_A$.

  (D3) It has at least $ 2500k(k+l)^2\log(2kl)$ in-neighbors in $\bigcup_{(i,j)\in L_4}V(P_{i,j}^{1,1}\cup P_{i,j}^{3,3})$.

 First we consider case (D1). Claim \ref{claim10} (iii) yields that $v$ satisfies (s2) and (s4).  We color $v$ with $I$, and $Z_A$ stays the same. Then we move to (D2).  We color $k$ of $v_1, v_2, \ldots , v_{2k}$ by $I$ and $k$ of them by $II$, and replace $Z_A$ by $Z_A\cup  \{v_1, v_2, \ldots , v_{2k}\}$. We also color $v$ with $I$. It follows from Property \ref{lemma4} (ii) and (iv) that all vertices in $\{v,v_1, v_2, \ldots , v_{2k}\}$ satisfy (s2) and (s4).

Finally, we turn to (D3). Recall that $|P_{i,j}^{3,3}|=2(k+l+2)$, so for all $(i,j)\in L_4$ let $(i,j)_l$ denotes the vertex in $V(P_{i,j}^{3,3})$ where $l\in [2k+2l+4]$. Moreover, write
$$P_{i,j}^{3,3}=(i,j)_{2k+2l+4}(i,j)_{2k+2l+3}\cdots (i,j)_{2}(i,j)_{1}.$$
Suppose that the vertex $(i,j)_{l_{i,j}}$ is the first vertex in $V(P_{i,j}^{3,3})$ such that $(i,j)_{l_{i,j}}v\in A(D)$ along the orientation of $P_{i,j}^{3,3}$. Claim \ref{claim9} (i) states that the color of the vertices in $V(P_{i,j}^{3,3})$  alternates along the orientation of $P_{i,j}^{3,3}$. Since $\delta(D)\geq n-l$, it follows from (A2) that there are at least $\lfloor l_{i,j}/2\rfloor -l-2$ vertices of color $I$ belonging to the in-neighborhood of $v$, and at least $\lfloor l_{i,j}/2\rfloor -l-2$ vertices of color $II$ belonging to the in-neighborhood of $v$. We see that if
$\sum_{(i,j)\in L_4}(\lfloor l_{i,j}/2\rfloor -l-2)\geq k$, then by Claim \ref{claim9} (ii) and Property \ref{lemma4} (iii)-(iv), the vertex $v$ satisfies (s2) and (s4). Thus $\sum_{(i,j)\in L_4}(\lfloor l_{i,j}/2\rfloor -l-2)\leq k-1$. This implies
$\sum_{(i,j)\in L_4}l_{i,j}<  2500(k+l)f(k,l)$. That is, the in-degree of $v$ in $\bigcup_{(i,j)\in L_4}V(P_{i,j}^{3,3})$ is less than $2500k(k+l)^2\log(2kl)$. Therefore, $v$ has at least an in-neighbors in $\bigcup_{(i,j)\in L_4}V(P_{i,j}^{1,1})$.  By Claim \ref{claim9} (iii), we color $v$ with $I$, and $v$ satisfies (s2) and (s4). Then $Z_A$ stays the same.

Note that we add at most $2k$ vertices to $Z_A$ for each vertex $v\in E_A \setminus (C_5\cup W)$.  So at the end of Step 1, we will still have $|Z_A| \leq  2k|E_A|$.

\textbf{Step 2.} We can color each vertex in $E_B \setminus (C_5\cup W)$ in turn as well as some set $Z_B$ of additional vertices in such a way that all the vertices in $E_B \setminus (C_5\cup W)$ satisfy (s1) and (s3). At first, let $Z_B =\emptyset$. At each substep, we
will update $Z_B$ such that

($a^\prime$) $Z_B$ consists of colored vertices and $Z_B\cap (C_5 \cup E_B) = \emptyset$,

 ($b^\prime$) every colored vertex lies in $C_5 \cup E_B \cup  Z_B$,

 ($c^\prime$) $|Z_B| \leq  2k|E_B|$ and so $|Z| \leq  4k|E_B|$, where $Z=Z_A \cup Z_B$.

 Now suppose that we have already colored some vertices in $E_B \setminus (C_5\cup W)$ such that $Z_B$ satisfies ($a^\prime$)-($c^\prime$). Combining this with (\ref{1}), (\ref{2}) and Claim \ref{claim10} (iii) implies that
 \begin{align}\label{8}
 &\hat{\delta}^+(D)-|E\cup Z|- |C_5\setminus (\bigcup_{(i,j)\in L_3}V(P_{i,j}^2) \cup \bigcup_{(i,j)\in L_4}V(P_{i,j}^{1,1}\cup P_{i,j}^{3,3}))|\nonumber \\
 \geq &\hat{\delta}^+(D) -5k|E| - |C_5\setminus (\bigcup_{(i,j)\in L_3}V(P_{i,j}^2) \cup \bigcup_{(i,j)\in L_4}V(P_{i,j}^{1,1}\cup P_{i,j}^{3,3}))|\nonumber\\
 \geq &2500(k+l)f(k,l)+ 2k.
 \end{align}
 Since all colored vertices lie in $C_5\cup  E \cup Z$, eq.(\ref{8}) implies that for a vertex $v\in E_B \setminus (C_5\cup W)$, it satisfies one of the following statements.

 (D$1^\prime$) It has at least an out-neighbor in $\bigcup_{(i,j)\in L_3}V(P_{i,j}^2)$.

 (D$2^\prime$) It has at least $2k$ uncolored out-neighbors $v_1, v_2, \ldots , v_{2k}$ outside $E$.

  (D$3^\prime$) It has at least $ 2500k(k+l)^2\log(2kl)$ out-neighbors in $\bigcup_{(i,j)\in L_4}V(P_{i,j}^{1,1}\cup P_{i,j}^{3,3})$.

The cases (D$1^\prime$)-(D$3^\prime$) can be handled in much the same way as (D1)-(D3). Note that we add at most $2k$ vertices to $Z_B$ for each vertex in $E_B \setminus (C_5\cup W)$. Thus we always have $|Z_B|\leq 2k|E_B|$ and so $|Z| \leq 4k|E|$. After Step 2, all colored vertices in $C_5\cup E\setminus W $ are safe, and
 \begin{equation*}
   \begin{aligned}
   |C_6|&\leq |C_5|+|E|+|Z|\overset{(\ref{1}), (\ref{2})}{\leq}\frac{n}{ 50}+\frac{n}{ 40}+\frac{\hat{\delta}^+(D)}{ 1000k}+\frac{4\hat{\delta}^+(D)}{1000}\leq \frac{9n}{ 200}+\frac{\hat{\delta}^+(D)}{200}\leq \frac{n}{ 20}.
   \end{aligned}
 \end{equation*}
This completes the proof of Claim \ref{claim11}.
\end{proof}
\subsection{Construction of sets $V_1$ and $V_2$}
In this subsection, we will construct sets $V_1$ and $V_2$ as Theorem \ref{theorem1} required.

%$|V_I|=a_1$, and $|V_{II}|=a_2$
%
%
%If there is a correct $i$-path $P_i$ in $D-D_0$ for each $i\in[6k]$, and each vertex of D is safe, then the sets $V_1= V_I$ and $ V_2 = V_{II}$ form a partition of $V (D)$ as required in Theorem \ref{theorem1}.

\begin{claim}\label{f}
We may color some additional vertices such that the sets $V_1:= V_I$ and $ V_2: = V_{II}$ as required in Theorem \ref{theorem1}.
\end{claim}
\begin{proof}
According to Claim \ref{claim11}, the number of all colored vertices is at most  $n/20$. It follows from Property \ref{lemma4} (i) and Claim \ref{claim12} that all uncolored vertices will be safe with whichever color. This implies that we may color some additional vertices such that $|V_I|=n_1 \geq n/20$, $|V_{II}|=n_2\geq n/20$, and all colored vertices are safe. Therefore, it remains to show the connectivity of $D[V_I]$, $D[V_{II}]$ and $D[V_I,V_{II}]$.

We first show that $D[V_I]$ is strongly $k$-connected. We need to
check that $D[V_I\setminus  S]$ contains a path from $x$ to $y$ for any set $S$ with $|S|\leq k-1$ and any two vertices $x, y \in V_I\setminus S$. Since each vertex is safe, (s1) implies that there
exists a path $P_x$ in $D[V_I\setminus  S]$ from $x$ to some vertex $x^\prime\in  V_I\setminus  (D_0 \cup E_B \cup S)$, and (s2) implies that there exists a path $P_y$ in $D[V_I\setminus  S]$ from some
vertex $y^\prime\in  V_I\setminus  (D_0 \cup E_A \cup S)$ to $y$. Without loss of generality assume $S\cap (A_1\cup B_1\cup V (P_1))=\emptyset$ (recall that $D[A_1]$ has a spanning transitive tournament with head $x_1$ and tail $a_1$, $D[B_1]$ has a spanning transitive tournament with tail
$y_1$ and head $b_1$, and $P_1$ is a path from $b_1$ to $a_1$). Since $x^\prime\notin D_0 \cup E_B$ and $y^\prime\notin D_0 \cup E_A$, (P3) implies that $x^\prime$ sends an arc to $x^{\prime\prime}\in B_1\setminus\{b_1\}$, and $y^\prime$ receives an arc from $y^{\prime\prime}\in A_1\setminus\{a_1\}$. Altogether this yields that
$D[V(P_x\cup P_y\cup P_1)\cup A_1\cup B_1]\subseteq  D[V_I\setminus  S]$ contains a path $P_xx^{\prime\prime}P_1y^{\prime\prime}P_y$, as desired. A similar argument shows that $D[V_{II}]$ is strongly $k$-connected too.

It remains to show
that $D[V_I, V_{II}]$ is strongly $k$-connected. We will show that $D[V_I\setminus  S,V_{II}\setminus  S]$ contains a path from $x$ to $y$ for any set $S$ with $|S|\leq k-1$ and any two vertices $x, y \in V(D)\setminus S$. Because each vertex is safe, (s3) yields that there exists a path $P_x$ in $D[V_I\setminus  S,V_{II}\setminus  S]$ from $x$ to some vertex
$x^\prime\in  V(T)\setminus  (D_0 \cup E_B \cup S)$, and (s4) yields that there exists a path $P_y$ in $D[V_I\setminus  S,V_{II}\setminus  S]$ from some
vertex $y^\prime\in  V(D)\setminus  (D_0 \cup E_A\cup S )$ to $y$. We now choose an index $i$ as follows:

(i) If $x^\prime, y^\prime\in  V_I$, let $i \in [2k + 1, 3k]$ be such that $S\cap (A_i \cup B_i \cup  V (P_i))=\emptyset$.

(ii) If $x^\prime, y^\prime\in  V_{II}$, let $i\in [3k + 1, 4k]$ be such that $S\cap (A_i \cup B_i \cup  V (P_i))=\emptyset$.

(iii) If $x^\prime\in  V_I$ and $y^\prime\in  V_{II}$, let $i \in [4k+1, 5k]$ be such that  $S\cap (A_i \cup B_i \cup  V (P_i))=\emptyset$.

(iv)  If $x^\prime\in  V_{II}$ and $y^\prime\in  V_I$, let $i \in [5k+1, 6k]$ be such that $S\cap (A_i \cup B_i \cup  V (P_i))=\emptyset$.

Note that $x^\prime\notin D_0 \cup E_B$ and $y^\prime\notin D_0 \cup E_A$, (P3) implies that $x^\prime$ sends an arc to $x^{\prime\prime}\in B_i\setminus\{b_i\}$, and $y^\prime$ receives an arc from $y^{\prime\prime}\in A_i\setminus\{a_i\}$. According to (i)-(iv), the digraph $D[V_I\setminus  S,V_{II}\setminus  S]$ contains a path $P_xx^{\prime\prime}P_iy^{\prime\prime}P_y$, as desired.
\end{proof}

\section{Concluding remark}
In this section, we obtain that a strongly $\Omega(k^3l^2t^2\log(2kl))$-connected digraph of order $n$ with $\delta(D)\geq n-l$ admits a vertex partition into $t$ disjoint sets $V_1,V_2,\ldots,V_t$ such that each of $D [V_1], \ldots,D[V_t]$, and $D[V_1,\ldots,V_t]$ is strongly $k$-connected. This result can be handled in much the same way as Theorem \ref{theorem1}, the only difference
being in the construction and coloring of almost in/out-dominating sets. To prove this result, we need to  construct $(t+4)k$  almost out-dominating sets $A_1,\ldots, A_{(t+4)k}$ and $(t+4)k$  almost in-dominating sets $B_1,\ldots, B_{(t+4)k}$ satisfy (P0)-(P4).  For each $j\in [t]$, define $j^*=[(j-1)k+1,(j-1)k+k]$. Set
$D_j= \bigcup_{i\in [j^*]}(A_i\cup B_i)$, $\bigcup\limits_{i\in[tk+1,(t+1)k]}
\{a_i,b_i\}\cup  \bigcup\limits_{i\in[(t+2)k+1,(t+3)k]}
\{b_i\}  \cup \bigcup\limits_{i\in[(t+3)k+1,(t+4)k]}
\{a_i\} \subseteq D_1$ and $\bigcup\limits_{i\in[(t+1)k+1,(t+2)k]}
( \{a_i,b_i\}) \cup \bigcup\limits_{i\in[(t+2)k+1,(t+3)k]}
\{a_i\} \cup \bigcup\limits_{i\in[(t+3)k+1,(t+4)k]}
\{b_i\}\subseteq D_2$ (See figure \ref{fig1} for $t=2$). Then we can color all vertices in $D_j$ with color $j$ and some other vertices such that all colored vertices are safe. The rest of the proof runs as the proof of Theorem \ref{theorem1}.

Inspired by the results of Gir\~{a} and Letzter \cite{An(2022)}, and  Kang and Kim  \cite{Kang(2020)}, it is natural to consider the following questions.

\begin{conjecture}
There exists a constant $c>0$ such that the vertices of every strongly $c\cdot ktl$-connected digraph of order $n$  with $\delta(D)\geq n-l$ can be partitioned into $t$ parts $V_1, V_2,\ldots, V_t$, each of $D [V_1], \ldots,D[V_t]$ and $D[V_1,\ldots,V_t]$ is a strongly $k$-connected digraph.
\end{conjecture}

\begin{conjecture}
Let $k, t,l, m, n, q, a_1,\ldots, a_t \in \mathbb{N}$ with $t, m \geq 2$, $\sum_{i\in [t]} a_i \leq n$ and $a_i\geq n/(10tm)$ for each $i \in [t]$. Suppose that D is a strongly $c\cdot ktlm$-connected digraph of order n  with $\delta(D)\geq n-l$, and $Q_1,\ldots, Q_t \subseteq V (D)$ are t disjoint sets with $|Q_i|\leq q$ for each $i\in [t]$. Then there exist disjoint sets $W_1, \ldots, W_t$ of $V(D)$ satisfying the following.

\noindent(1) $Q_i \subseteq W_i$.

\noindent (2) Each of $D[W_1], \ldots, D[W_t]$ and $D[W_1,\ldots,W_t]$ is strongly $k$-connected.

\noindent (3) $|W_i| = a_i$.
\end{conjecture}

\section{Acknowledgements}
We are extremely grateful to the anonymous referees for their careful reading and suggestions.

\end{document}